\documentclass{amsart}
\usepackage{a4,amssymb,amsfonts,amsmath,enumerate,color,url}
\usepackage{amsthm}	
\usepackage{mathbbol}

\usepackage[utf8]{inputenc}

\def\:{\thinspace:\thinspace}
\newtheorem{theo}{Theorem}
\newtheorem{lemma}[theo]{Lemma}
\newtheorem{prop}[theo]{Proposition}
\newtheorem{cor}[theo]{Corollary}

\newtheorem{rem}[theo]{Remark}

\numberwithin{equation}{section}
\numberwithin{theo}{section}
\def\:{\thinspace:\thinspace}
\theoremstyle{definition}

\numberwithin{theo}{section}
\DeclareMathOperator{\Ker}{Ker}
\DeclareMathOperator{\id}{id}
\DeclareMathOperator{\Id}{Id}
\DeclareMathOperator{\Real}{Re}
\DeclareMathOperator{\Ima}{Im}

\def\@secappcntformat#1{%
 \ifappendix \rm\appendixname\ifoneappendix\else~\fi\fi 
 \ifoneappendix \else \csname the#1\endcsname\relax\fi
 \ifx\apphe@d\@empty \else .\fi\enskip 
}
 \author{Delio Mugnolo}
 
 \address{Delio Mugnolo, Institut f\"ur Analysis, Universit\"at Ulm, 89069 Ulm, Germany}
\email{delio.mugnolo@uni-ulm.de}

\author{Serge Nicaise}

\address{Serge Nicaise, Universit\'e de Valenciennes et du Hainaut Cambr\'esis, LAMAV, FR CNRS 2956, ISTV, F-59313 - Valenciennes Cedex 9, France}
\email{Serge.Nicaise@univ-valenciennes.fr}

\title[Heat and wave equations with non-local conditions]{Well-posedness and spectral properties of\\ heat and wave equations with non-local conditions}
\subjclass[2000]{47D06, 35J20, 34B10}
\keywords{Non-local conditions; Analytic semigroups; Quadratic forms; Weyl asymptotics}

\begin{document}

\begin{abstract}
We consider the one-dimensional heat and wave equations but -- instead of boundary conditions~-- we impose on the solution certain non-local, integral constraints. An appropriate Hilbert setting leads to an integration-by-parts formula in Sobolev spaces of negative order and eventually allows us to use semigroup theory leading to analytic well-posedness, hence sharpening regularity results previously obtained by other authors. In doing so we introduce a parametrization of such integral conditions that includes known cases but also shows the connection with more usual boundary conditions, like periodic ones. In the self-adjoint case, we even obtain eigenvalue asymptotics of so-called Weyl's type.
\end{abstract}

\maketitle

\section{Introduction}

Several problems in the applied sciences are modeled by partial differential equations on domains that, as such, require the modeler to impose some assumption on the sought-after solution in order to obtain well-posedness in some function space. Typical are, of course, boundary conditions like Dirichlet or Neumann ones. However, in many physical problems certain different constraints are natural: for example, all equations that are derived from conservation laws -- like the Cahn--Hilliard equation or the Navier--Stokes equation -- admit the conservation of a certain physical quantity (e.g.\ mass, barycenter, energy, or momentum) and it is natural to wonder whether already these \emph{minimal} constraints suffice to obtain well-posedness, at least for a special choice of the initial conditions. While this idea seems to be widely applicable, for simplicity  we restrict this paper to the heat and the wave equations on an interval. Instead of usual boundary conditions on the solution $u$, we investigate the role of non-local, integral conditions like
\begin{equation}\label{bouzzcond}
\int_0^1 u(t,x)dx=0,\qquad t\ge 0,
\end{equation}
This amounts to imposing that the moment of order 0 (corresponding e.g.\ to the total mass, in the case of a diffusion equation for which $u$ denotes the relative density of a mixture) vanishes identically.
A heat equation complemented by the above condition has been introduced by J.R.~Cannon in~\cite{Can63}, where well-posedness was investigated by methods based on abstract Volterra equations. While Cannon's work has received much attention by numerical analysts, it has gone largely overlooked by the PDE community, with some notable exceptions (cf.~\cite{IonMoi77,Yur86,Shi93} and the references therein to earlier Soviet literature). The fact that~\eqref{bouzzcond} only eliminates one degree of freedom still forced Cannon and later investigators to impose a local (say, Dirichlet) condition at one of the endpoints.

More recently, it has been observed that the local condition at 0 or 1 may be dropped and replaced by another condition on the moment of order 1, like
\begin{equation}\label{bouzzcond2}
\int_0^1 (1-x)u(t,x)dx=0,\qquad t\ge 0.
\end{equation}
Wave and heat equations with (generalizations of) conditions~\eqref{bouzzcond} and~\eqref{bouzzcond2} have been intensively studied by A.~Bouziani and L.S.~Pul'kina in a long series of papers that seems to begin with~\cite{Pul92} and~\cite{BouBen96}. In~\cite{Bou02b}, a condition on the moment of order 2 is discussed. In fact, over the last 20 years Bouziani, Pul'kina and their coauthors have discussed a manifold of hyperbolic, parabolic and pseudoparabolic equations with such conditions, mostly by numerical methods. Among others, in~\cite{Bou04,DaiHua07,GueBel09} several weaker well-posedness results for related parabolic, hyperbolic or pseudoparabolic equations have been obtained by different methods.
An extensive list of further papers treating these or similar conditions can be found in the introduction of~\cite{Deh05}. A few tentative extensions of the above conditions for heat or wave equations on higher dimensional domains have been proposed in the literature, cf.~\cite{MesBou01,Pul07}. We are not aware of earlier investigations about the possibility of replacing a condition on the moment of order one by a condition on some moment of higher order. In the companion papers~\cite{MugNic13b,MugNic14} we have further developed our techniques in order to address these issues.

\bigskip
The main goal of this article is to provide an abstract framework -- as general as possible --for studying the one-dimensional heat or wave equation with integral conditions by means of semigroup theory.
{It turns out that, for the spatial operator we are considering, the associated diffusion equation is the gradient flow} of a very simple functional -- up to lower order terms, it is simply the $L^2$-norm -- with respect to some $H^{-1}$-type inner product. This is not surprising: for example, it is well-known that the gradient flow associated to the $L^p$-norm with respect to the $H^{-1}(0,1)$-inner product is the porous medium equation with Dirichlet boundary conditions. Also the observation that replacing $H^{-1}(0,1)$ by $H^{-1}(T)$ permits us to realize the diffusion equation with integral conditions as a gradient flow is not entirely new: for example, it is implicitly used in several articles by F.\ Otto (see e.g.~\cite{KohOtt02}) to discuss some modifications of the Cahn--Hilliard equation. However, by some direct computations performed in Section~\ref{sec:homog} it turns out that the $H^{-1}(T)$-norm, although natural, gives rise to a gradient flow that agrees with the usual diffusion equation only in a suitable quotient space. This surprising fact seems to show that the usual second derivative is not suitable to be endowed with (nonlocal) moment conditions. Indeed, also the associated second order (in time) evolution equation is not the classical wave equation, but rather a non-trivial generalization that coincides with the usual one only for smooth initial data, cf.\ Theorem~\ref{wellp2}. In the case of the heat equation, instead, this phenomenon can actually be overcome by invoking the smoothing properties of the analytic semigroups yielded by our approach.

The main difficulty associated with our variational approach is that less regular functions suffer a dramatical deterioration of their properties when integrated by parts even against smooth functions. In turn, this yields that the evolution equation associated with the above mentioned gradient flow is, seemingly, an exotic parabolic equation with no evident physical interpretation, see e.g.\ Theorem~\ref{identK3}. In order to get rid of these effects, in all the above mentioned papers (see e.g.~\cite{CanEstvan87,BouMerBen08}) the initial data were assumed to have an artificially strong regularity (and the solutions were only shown to satisfy the equation in a weak sense). Our semigroup approach allows us to avoid this problem and to obtain a classical solution even for rough initial data, cf.~Theorem~\ref{wellp2}.

Let us finally emphasize one of the advantages of our approach: It is known that as soon as an evolution equation with homogeneous boundary conditions is governed by a $C_0$-semigroup, one can extend the functional setting to find a solution to the same evolution equation with \emph{inhomogeneous} boundary conditions that is still given by a $C_0$-semigroup. This classical method is recalled in Remark~\ref{rem:inhomog} and shows in turn why it essentially suffices to focus on the condition~\eqref{bouzzcond}, whose interpretation as a condition on the \emph{mass} may appear less physical at a first glance, as it implies the existence of regions with negative density.

\bigskip
The paper is organized as follows. In Section~\ref{sec:bouz} we introduce the basic notations and prove some technical lemmata along with an integration-by-parts-type formula that will prove quite useful, since we are forced to work in Sobolev spaces of negative order so that the usual Gau{\ss}--Green formulae do not apply immediately.

We further discuss a class of quasi-accretive extensions of the second derivative with non-local constraints, {but due to technical reasons we have to tackle two subcases}: the analysis of the corresponding parabolic and hyperbolic problems will be performed in Section~\ref{sec:homog}. By introducing two parameters (a subspace $Y$ of $\mathbb C^2$ and a $2\times2$-matrix $K$, respectively) we are able to treat an infinite class of non-local constraints.

In some particular cases, the spatial operator is even self-adjoint. In Section~\ref{sec:spectral} we are able to describe the spectrum of several self-adjoint extensions of the ``minimal'' second order derivative, by more or less elementary techniques. We can show among other things that eigenvalue asymptotics of so-called Weyl's type is still valid for such a  class of operators. 

We end up with some direct extensions of our results to the case of dynamic integral conditions, see section~\ref{sec6}, which in turn by known semigroup theoretical methods allow us to treat \emph{inhomogeneous} integral conditions.

\bigskip
As observed above, it seems that following the original article by Cannon, in the literature the attention has been devoted mostly to problems in which mixed boundary/integral conditions are considered, see e.g.~\cite{Shi93} as well as~\cite[\S7.5]{Can84} and references therein. It should be emphasized that this kind of problems is slightly different from ours: it seems impossible to fix our parameters $Y,K$ in such a way that a homogeneous Dirichlet condition at one endpoint can be recovered. Nor we are able to treat generalizations of~\eqref{bouzzcond} that have enjoyed some popularity in the literature, like \emph{time-dependent}, mixed boundary/integral constraints of the form
\begin{equation*}\label{bouzzcondplus}
u(t,0)=0\qquad \hbox{and}\qquad \int_0^{b(t)} u(t,x)dx=0,\qquad t\ge 0,
\end{equation*}
for some given function $b$ that satisfies a suitable smoothness assumption, as in the original paper by Cannon~\cite{Can63}, or even
\begin{equation*}\label{bouzzcondplus2}
\int_0^{b(t)} u(t,x)dx=\int_{b(t)}^1 u(t,x)dx=0,\qquad t\ge 0,
\end{equation*}
as considered in~\cite{LesEllIng98}.

Conversely, it seems that the methods in the quoted articles cannot be adapted to our general setting. Furthermore, we could not find any previous reference to a classification of self-adjoint integral conditions as the one we perform in Section~\ref{sec:homog}. Most importantly, our operator theoretical approach is based on energy methods and $C_0$-semigroups. To the best of our knowledge, it is the first time these kinds of problems are studied in this framework.

Our technique presents quite a few advantages over alternative methods. First of all, we deliver a unified approach that allows us to treat simultaneously a whole class of integral constraints (including those in~\eqref{bouzzcond}-\eqref{bouzzcond2}); secondly, it suffices for us to prove well-posedness of the undamped wave equation to obtain automatically, by perturbation methods and the general theory of $C_0$-semigroups, well-posedness of, among others, the telegraph equation and the heat equation. We are also able to enlarge the space on which well-posedness is given, in comparison with previous literature: e.g., in~\cite{Bou04} the author needs to impose a compatibility condition on the first moment of the initial value, which we can instead drop. 

However, the most important by-product of the semigroup approach pursued by us is that valuable information becomes available about the solution operators, in comparison with other techniques: e.g., we obtain analyticity of the semigroup that governs the heat equation. This in turn yields the fact that the solution $u$ is smooth with respect to both the time and space variables and automatically satisfies additional boundary conditions, cf.\ Corollary~\ref{regulsol}. This should be compared with the much weaker regularity results obtained in~\cite[\S3]{CanEstvan87} or~\cite[\S5]{BouMerBen08} by means of a Galerkin method. Moreover, in this way we are also able to prove that solutions of the heat equation automatically satisfy infinitely many (non-local) \emph{boundary} conditions along with the integral ones. Also this observation, which is actually a straightforward consequence of our semigroup approach, seems to be new.

Finally, let us stress that form methods are often very efficient because they allow for an easy proof of further properties. In particular, one is often able to show via the so-called Beurling--Deny conditions that the semigroup, which is a priori only generated in a Hilbert space, actually extrapolates to a range of $L^p$-spaces. It seems however out that our functional framework does not allow for an effective application of the the Beurling--Deny method. A rather different approach has been developed by A.\ Bobrowski and the first author in order to turn our integral conditions into boundary ones: In this way, it is proved in~\cite{BobMug13} that the heat equation is well-posed in all $L^p(0,1)$-spaces, $p\in [1,\infty]$, and even in $C[0,1]$, if conservation of moments of order 0 and 1 is imposed.

{\bf Added in proof:} After the manuscript of this article was submitted we have learned from Fritz Gesztesy that some of the operators studied by us fall into the scope of Krein's theory of self-adjoint extensions, cf.~\cite{MugNic14} for details. In particular, our spectral results in Corollary~\ref{cWeylgenY}.(i) are weaker than those obtained in~\cite[\S~5]{AloSim80}, cf.~\cite[\S~6]{MugNic14} for further details.

\section{The Bouziani space}\label{sec:bouz}

If we consider $(0,1)$ as the torus $T$, then the test function set ${\mathcal D}(T)$ is in fact the set of smooth functions in $[0,1]$
such that the derivatives at all orders coincide at 0 and 1. In the same manner we will use the Sobolev space $H^1(T)$, by which we denote the subspace of those $u\in H^1(0,1)$ such that $u(0)=u(1)$ (i.e., of those $H^1$-functions supported on the torus). We denote by $H^{-1}(T)$ its dual.
{
In view of the decomposition 
\[
H^1(0,1)\ni \quad u\;\equiv\; \big(u-u(0)\id-u(1)1\big)\oplus u(0)\id +u(1)1\quad \in H^1_0(0,1)\oplus {\mathbb C}\id\oplus\; {\mathbb C}1
\]
where $\id$ is the identical function, one sees that $H^1_0(0,1)$ is a closed subspace of codimension 1 of $H^1(T)$ as well as a closed subspace of codimension 2 of $H^1(0,1)$. In particular, $H^1(T)=H^1_0(0,1)\oplus {\mathbb C}1$ and hence $H^1_0(0,1)$ is not dense in $H^1(T)$. Therefore, using $L^2(0,1)$ as a pivot space, it is not true that $H^{-1}(T)$ is continuously embedded in $H^{-1}(0,1)$. Indeed, the Dirac functional $\delta_1$ lies in $H^{-1}(T)$ and is not trivial there, but it agrees with the 0 functional in $H^{-1}(0,1)$. Hence, by duality $H^{-1}(T)$ can be identified with $H^{-1}(0,1)\oplus {\mathbb C}$ with respect to the orthogonal decomposition (for the inner product~\eqref{prodh1} defined below)
\begin{equation}\label{h1decomp}
H^{-1}(T)\ni \quad \phi\; \equiv\; \big(\phi-\mu_0(\phi)\delta_1)\oplus \mu_0(\phi)\delta_1\in H^{-1}(0,1)\oplus {\mathbb C} \delta_1.
\end{equation}
We thus regard $H^{-1}(0,1)$ as a closed subspace of $H^{-1}(T)$. Here $\mu_0$ is the linear functional on $\mathcal D'(T)$ defined by
\begin{equation}
\mu_0(\varphi):=\langle \varphi,1\rangle,\qquad \varphi\in \mathcal D'(T),\label{mu0mu1}
\end{equation}
which is  bounded on $H^{-1}(T)$. We will denote throughout by $\Id$ the orthogonal projection of $H^{-1}(T)$ onto $H^{-1}(0,1)$, and by $\Id_m$ its restriction to 
\begin{equation}\label{Hdefin}
H:=\{w\in H^{-1}(T)\, :\, \mu_0(w)=0\},
\end{equation}
which clearly is an isometric isomorphism. 
Observe that $\mu_0(u)$ is simply the mean value of $u$, whenever $u\in L^1(0,1)$, and with an abuse of terminology we will in fact say that $\phi\in {\mathcal D}'(T)$ has \emph{mean zero} whenever $\mu_0(\phi)=0$.
Observe that by construction
\begin{equation}\label{leqiso}
\Ker \Id ={\rm Span }\,\delta_1\ .
\end{equation}
}

Now for $\varphi\in {\mathcal D}(T)$, we can define a primitive of
$-\varphi+\int_0^1 \varphi(z)dz$
by
$$
J\varphi(x):=\int_x^1 \left(\varphi(y)-\int_0^1 \varphi(z)dz\right)dy,\qquad x\in (0,1),
$$
and therefore for $u\in {\mathcal D}'(T)$ we define its primitive $Pu\in {\mathcal D}'(T)$ by
\begin{equation}\label{0}
\langle Pu, \varphi\rangle:=\langle u, J\varphi\rangle,\qquad \varphi\in {\mathcal D}(T).
\end{equation}
Note that 
\begin{equation}\label{sergeajout2}
\langle Pu, 1\rangle=0,\qquad \forall u\in {\mathcal D}'(T).
\end{equation}
This means that $Pu$ is the unique primitive of $u$ of zero mean.

Now, by definition
$\langle (Pu)', \varphi\rangle=-\langle Pu, \varphi'\rangle=-\langle u, J\varphi'\rangle$.
But since $\varphi$ is periodic,
$$ J\varphi'(x)= \int_x^1\varphi'(y)dy= \varphi(1)- \varphi(x)\qquad\forall \varphi\in {\mathcal D}(T),$$
or rather
$$\langle (Pu)', \varphi\rangle=\langle u, \varphi\rangle-\varphi(1) \langle u, 1\rangle\qquad\forall \varphi\in {\mathcal D}(T).$$
Hence the identity
\begin{equation}\label{sergeajout1}
(Pu)'=u-\langle u,1\rangle \delta_1,\qquad \forall u\in {\mathcal D}'(T)
\end{equation}
holds. Although $(Pu)'$ is not necessarily equal to $u$, the name ``primitive of $u$'' for $Pu$ is justified by the fact that $(Pu)'=u$ on the space of distributions $u\in {\mathcal D}'(T)$ of zero mean and also on ${\mathcal D}'(0,1)$.
Note finally that
\begin{equation}\label{sergeajout3}
 P \delta_1=0 \qquad \hbox{in }{\mathcal D}'(T),
\end{equation}
since $\langle P\delta_1, \phi\rangle=\langle \delta_1,J\phi\rangle = J\phi(1)=0$
which is possibly surprising but is in accordance with~\eqref{sergeajout1}.

\begin{lemma}\label{lemma1}
Let $u\in L^{2}(0,1)$. Then $Pu$ is given by
\begin{equation}\label{2}
Pu(x)={\mathcal I} u(x)-\int_0^1 {\mathcal I} u(z) dz,\qquad \forall x\in (0,1),
\end{equation}
where
$$
{\mathcal I} u(x):=\int_0^x u(y)dy,\qquad x\in (0,1).
$$
In particular, $Pu\in H^1(0,1)$ and $(Pu)'=u$ in $ {\mathcal D}'(0,1)$.
Moreover, $P$ is bounded from $H^{-1}(T)$ to $L^2(0,1)$.
\end{lemma}
\begin{proof}
First we remark that for $u\in L^{2}(T)$ and $\varphi\in {\mathcal D}(T)$, we have
$$
\langle Pu, \varphi\rangle=\int_0^1 ({\mathcal I} u)'(x) J\varphi(x)\,dx.
$$
Hence by integration by parts, we obtain
$$
\langle Pu, \varphi\rangle=\int_0^1 ({\mathcal I} u)(x) \left(\varphi(x)-\int_0^1 \varphi(z)dz\right)\,dx,
$$
since the boundary terms vanish due to $({\mathcal I} u)(0)= J\varphi(1)=0$. 
This shows~\eqref{2} and that $Pu$ belongs to $L^{2}(T)$.

In a second step we first easily check that for $\varphi\in {\mathcal D}(T)$,
$J\varphi$ is in $H^1(T)$ (it is even in $H^1_0(0,1)$) and that
$$
\|J\varphi\|_{H^1_0(0,1)}\lesssim \|\varphi\|_{L^2}.
$$
According to~\eqref{0} for any $u\in H^{-1}(T)$, we then get
$$
| \langle Pu, \varphi\rangle|\leq \| u\|_{H^{-1}(T)} \|J\varphi\|_{H^1_0(0,1)}\lesssim \| u\|_{H^{-1}(T)} \|\varphi\|_{L^2}, \forall \varphi\in {\mathcal D}(T).
$$
As ${\mathcal D}(T)$ is dense in $L^2(0,1)$, this shows that
$$\|Pu\|_{L^2}\lesssim \|u\|_{H^{-1}(T)},\qquad \forall u\in H^{-1}(T),$$
as we wanted to prove.
\end{proof}

Let us now introduce the spaces
$$V:=\left\{f\in L^2(0,1):\mu_0(f)=\mu_1(f)=0\right\}$$
and
$$\tilde{V}:=\left\{f\in L^2(0,1):\mu_0(f)=0\right\},$$
where $\mu_0$ is defined in \eqref{mu0mu1} and
\begin{equation}
\mu_1(f):=\int_0^1(1-x)f(x)\,dx 
%=[(1-x)\int_0^x f(z)dz]_0^1 +\int_0^1 \int_0^x f(z)dz
=\int_0^1 \int_0^x f(z)dz\; dx.%=\langle \int_0^\cdot f(x) dx,1\rangle.
\label{mu0mu12}
\end{equation}
With this notation,~\eqref{2} reads 
\begin{equation}\label{Pmu1}
Pu(\cdot)={\mathcal I}u(\cdot)-\mu_1(u),\qquad \forall u\in L^2(0,1),
\end{equation}
and moreover
\begin{equation*}%\label{mu0=mu1}
\mu_0({\mathcal I}u)=\mu_1(u),\qquad \forall u\in L^2(0,1).
\end{equation*}
Note that $\tilde V$ (resp. $V$) is the orthogonal complement in $L^2(0,1)$ of 
${\mathbb P}_0(\mathbb R)$) (resp. ${\mathbb P}_1(\mathbb R)$). (Here and in the following we denote by ${\mathbb P}_n({\mathbb R})$ the space of polynomials of one real variable with complex coefficients and degree $\le n$.)

\begin{cor}\label{rkpf}
The linear operator $P$ is bounded from $H^{-1}(T)$ to $\tilde{V}$, from $\tilde{V}$ to $H^1(T)$, and from $V$ to $H^1_0(0,1)$.
\end{cor}
\begin{proof}
The first assertion directly follows from ~\eqref{sergeajout2} and the previous Lemma.
The second and third ones are a consequence of the properties
$$Pf(0)=-\mu_1(f)\qquad \hbox{and}\qquad Pf(1)=\mu_0(f)-\mu_1(f),$$
that are valid for any $f\in L^2(0,1)$.
\end{proof}

\begin{lemma}\label{lemma2}
We have 
\begin{equation*}\label{1}
 \|u\|_{H^{-1}(T)} \lesssim \|Pu\|_{L^2}+|\langle u, 1\rangle|, \qquad u\in H^{-1}(T).
\end{equation*}
\end{lemma}
\begin{proof}
For all $u\in H^{-1}(T)$ we have by~\eqref{sergeajout1}
\[ u=(Pu)'+\langle u, 1\rangle \delta_1 \quad\hbox{ in } H^{-1}(T).\]
Therefore, for $\psi\in {\mathcal D}(T)$, $\psi'\in {\mathcal D}(T)$ and has mean zero,
and hence
$J\psi'(x)=-\psi(x)+ \psi(1)$. Accordingly, by~\eqref{0}
$$
\langle Pu, \psi'\rangle=-\langle u, \psi-\psi(1)\rangle=-\langle u, \psi\rangle+\psi(1)\langle u,1\rangle,
$$
or equivalently
$$
\langle u, \psi\rangle=-\langle Pu, \psi'\rangle+\psi(1)\langle u,1\rangle=\langle D Pu, \psi\rangle+\psi(1)\langle u,1\rangle.
$$
That implies
$$
|\langle u, \psi\rangle|\leq |\langle Pu, \psi'\rangle|+|\psi(1)||\langle u,1\rangle|.
$$
By the Sobolev embedding theorem, we obtain
 $$
|\langle u, \psi\rangle|\lesssim (\| Pu\|_{L^2} +|\langle u,1\rangle|) \|\psi\|_{H^{1}(T)},\qquad \forall u\in H^{-1}(T),\; \psi \in {\mathcal D}(T).
$$
This leads to the conclusion by density.
\end{proof}

\begin{lemma}\label{lemmasecondderivative}
For all $f\in H^1(0,1)$ and any $c\in \mathbb C$, we have
$$
P(\Id _m^{-1}(f'')+c\delta_1)=f'-f(1)+f(0) 
$$
as an equality of $L^2$-functions.
\end{lemma}
\begin{proof}
For shortness we write $g:=\Id _m^{-1}(f'')$, which $g$ is well defined because $f''\in H^{-1}(0,1)$. By~\eqref{sergeajout1}
we have
$$
(Pg)'=g \quad\hbox{ in }{\mathcal D}'(T).
$$
By the definition of $\Id _m$ we have
$$
f''=g \quad\hbox{ in }{\mathcal D}'(0,1),
$$
and we deduce that
$$
(f'-Pg)'=0 \quad\hbox{ in }{\mathcal D}'(0,1).
$$
Hence there exists a constant $a\in \mathbb C$ such that
$$
f'-Pg=a \quad\hbox{ in } L^2(0,1).
$$
The conclusion follows from~\eqref{sergeajout2} and~\eqref{sergeajout3}.
\end{proof}

\begin{rem}\label{noidm}
Note that for $f\in H^2(0,1)$, we have
\begin{equation}\label{ajoutserge10}
\Id _m^{-1}(f'')=f''+(f'(0)-f'(1))\delta_1 \quad\hbox{ in } H^{-1}(T),
\end{equation}
and therefore
$$
P(\Id _m^{-1}(f''))=P(f'') \quad\hbox{ in } L^2(0,1).
$$
Hence, for functions regular enough the annoying term $\Id _m^{-1}u''$ can be safely replaced by $u''$. Indeed, the same holds for $f\in H^s(0,1)$ for all $s>\frac{3}{2}$, since then $f''\in H^{s-2}(0,1)$ and $1\in H^r_0(0,1)=H^r(T)$ for all $r<\frac{1}{2}$.
\end{rem}

In the same spirit, we have the next equivalence.

\begin{lemma}\label{lemmasecondderivative2}
Let $f\in H^1(0,1)$. Then 
$\Id _m^{-1}(f'')$ belongs to $L^2(0,1)$
if and only if
$f\in H^2(0,1)$ and $f'(0)=f'(1)$.
\end{lemma}
\begin{proof}
If $f\in H^2(0,1)$ is such that $f'(0)=f'(1)$, then 
the property $\Id _m^{-1}(f'')\in L^2(0,1)$ directly follows from \eqref{ajoutserge10}.
For the converse implication, we notice that 
$$
f''=\Id _m^{-1}(f'') \quad\hbox{ in }{\mathcal D}'(0,1).
$$
But by assumption the right-hand side of this identity belongs to $L^2(0,1)$, and therefore $f$ belongs to $H^2(0,1)$.
By~\eqref{ajoutserge10}, we get that
$(f'(0)-f'(1))\delta_1$ has to belong to $L^2(0,1)$, that is only possible if $f'(0)-f'(1)=0$.
\end{proof}

By Lemma~\ref{lemma2} the sesquilinear form defined by
\begin{equation}
\label{prodH}
(f,g)\mapsto \int_0^1 Pf(x) P\bar g(x) dx,\qquad f,g\in H,
\end{equation}
is an equivalent inner product on $H$. For this reason, we will always endow $H$ with the inner product 
$$(\cdot|\cdot)_H:=(P\cdot|P\cdot)_{L^2}$$ 
introduced in~\eqref{prodH}.

For an arbitrary subspace $Y$ of $\mathbb C^2$ we define
$$V_Y:=\{f\in L^2(0,1): (\mu_0(f),\mu_1(f))\in Y\}.$$
Obviously we have the inclusion $V\subset V_Y$ for any $Y$
and the identities $V=V_{\{0\}^2}$ and $\tilde V=V_{\{0\}\times \mathbb C}$.

For shortness we introduce the linear maps $\Gamma_1: L^2(0,1)\to \mathbb C^2$
and $\Gamma_2:H^2(0,1)\to \mathbb C^2$, defined by
\begin{equation}\label{gammagrand}
\Gamma_1 u:= \begin{pmatrix}
\mu_0(u)\\ \mu_1(u)
\end{pmatrix},\qquad \Gamma_2 u:=\begin{pmatrix}
-\mu_0( u'') -u(1)\\ u(1)-u(0)
\end{pmatrix}.
\end{equation}

\begin{rem}\label{surjmu}
If $Y$ is an arbitrary subspace of $\mathbb C^2$, then the compact mapping
$$\begin{pmatrix}
P_Y \Gamma_1\\ P_{Y^\perp} \Gamma_2
\end{pmatrix}:H^2(0,1)\to \mathbb C^2$$
is surjective, as one can see considering polynomials. In particular there exists $u\in V_Y$.
\end{rem}

\begin{lemma}\label{lemmamu1notcontinuous}
The functional $\mu_1: L^2(0,1)\to \mathbb C$ can be continuously extended to $L^1(0,1)$, but not to $H^{-1}(T)$.
\end{lemma}
\begin{proof}
Assume that there exists $C>0$ such that
\begin{equation}
\label{mu1continuous}
|\mu_1(g)|\leq C \|g\|_{H^{-1}(T)}\quad \forall g\in L^2(0,1).
\end{equation}
By direct calculations, we see that
$$
1-x=\frac12+\frac{1}{\pi}\sum_{k=1}^\infty \frac{\sin (2k\pi x)}{k},
$$
where the convergence of the series is in $L^2(0,1)$.
Then for any $n\in {\mathbb N}^*$ and letting
$$
u_n:=2\pi\sum_{k=1}^n {\sin (2k\pi \cdot)}\in L^2(0,1)
$$
we would have
$$
\mu_1(u_n)=\sum_{k=1}^n  \frac{1}{k}.
$$
Moreover by the standard characterization of the $H^{-1}(T)$-norm, we have
$$
\|u_n\|_{H^{-1}(T)}\sim \left(\sum_{k=1}^n  \frac{1}{k^2}\right)^{\frac12}.
$$
Hence (\ref{mu1continuous}) cannot hold since $\|u_n\|_{H^{-1}(T)}$ is uniformly bounded but clearly $\lim\limits_{n\to \infty}\mu_1(u_n)=+\infty$.
\end{proof}

\begin{cor}\label{cor2.5} The following assertions hold.
\begin{enumerate}
\item 
The space $V$ and $\tilde V$ are dense in $H$ defined in~\eqref{Hdefin}.
\item 
Let $Y$ be a subspace of $\mathbb C^2$ such that $Y\not=\{0\}^2$ and $Y\not=\{0\}\times \mathbb C$. Then
$V_Y$ is dense in $H^{-1}(T)$.
\end{enumerate}
\end{cor}
\begin{proof}
(1) Clearly, $\tilde V$ is dense in $H$. Let us pass to the density of $V$ in $H$.
 By Lemma~\ref{lemmamu1notcontinuous}, $\mu_1$ is not continuous on the closed subspace $H$ of $H^{-1}(T)$, either. To conclude the proof, it suffices to observe that the null space of the restriction of $\mu_1$ to $\tilde V$ is $V$, which is then dense in $H$ by~\cite[Thm.~1.18]{Rud73}.

(2) Before proving the second assertion we observe that if $Y\not=\{0\}^2$ and $Y\not=\{0\}\times \mathbb C$, then either $Y={\mathbb C}^2$ or there exists $\alpha \in \mathbb C$ such that $Y$ is the set of all $(z_0,z_1)\in \mathbb C^2$ satisfying
\begin{equation}\label{carY}
 z_1=\alpha z_0.
\end{equation}

For the case $Y=\mathbb C^2$, $V_Y=L^2(0,1)$ and the density is immediate. In the second case we can notice that
$$
V_Y=V\oplus \hbox{ Span } \{g\},
$$
for some $g\in V_Y$ such that 
$$
\begin{pmatrix}
\mu_0(g) \\ \mu_1(g)
\end{pmatrix}
=
\begin{pmatrix}
1\\ \alpha
\end{pmatrix},
$$
whose existence is guaranteed by Remark~\ref{surjmu}.
As $\mu_0(g)=1$, we see that
$$
H^{-1}(T)=H\oplus \hbox{ Span } \{g\},
$$
and the density of $V_Y$ into $H^{-1}(T)$ directly follows the density of $V$ into $H$.
\end{proof}

While by Lemma~\ref{lemmamu1notcontinuous} $\mu_1$ is not bounded on $H^{-1}(T)$, a weaker continuity property does hold.

\begin{lemma}\label{lemmamu1weakcontinuous}
There exists $C>0$ such that
\begin{equation}
\label{mu1weakcontinuous}
|\mu_1(g)|^2\leq C \|g\|_{L^2} \|g\|_{H^{-1}(T)},\quad \forall g\in 
L^2(0,1).
\end{equation}
\end{lemma}
\begin{proof}
The following trace inequality is standard (see for instance \cite[comment 1.(iii) at p. 233]{Bre10}):
\begin{equation}
\label{traceineq}
|u(1)|^2\leq 2\sqrt{2} \|u\|_{L^2} \|u\|_{H^{1}(0,1)},\quad\forall u\in H^1(0,1).
\end{equation}
%Indeed if $v\in H^1(0,1)$ is such that $v(0)=0$, we can write
%$$
%|v(1)|^2=\int_0^1 (v^2)'(x)\, dx,
%$$
%and by Cauchy-Schwarz's inequality we get
%$$
%|v(1)|^2\leq \|v\|_{L^2} \|v'\|_{L^2}.
%$$
%
%For a general $u\in H^1(0,1)$, it suffices to apply the previous estimate to
%the function
%$v:x\to x u(x)$.

Now if $g\in L^2(0,1)$, then by Lemma~\ref{lemma1} $Pg \in H^1(0,1)$
and therefore applying~\eqref{traceineq} to $Pg$, we get
$$
|(Pg)(1)|^2\leq 2\sqrt{2} \|Pg\|_{L^2} \|Pg\|_{H^{1}(0,1)}.
$$
Now, as Lemma~\ref{lemma1} shows that $(Pg)'=g$, we have
$$\|Pg\|_{H^{1}(0,1)}^2=\|Pg\|_{L^2}^2+\|g\|_{L^2}^2.$$
Moreover
$Pg$ being given by~\eqref{2}, 
 by the Cauchy--Schwarz inequality we see that
\begin{eqnarray*}
\|Pg\|_{L^2}\leq \|{\mathcal I} g\|_{L^2}+\left|\int_0^1\left(\int_0^x g(y)dy\right)dx\right|
\leq 2\|{\mathcal I} g\|_{L^2}
\leq \sqrt{2}\| g\|_{L^2}.
\end{eqnarray*}
Hence again owing to Lemma~\ref{lemma1} we obtain
$$
|(Pg)(1)|^2\leq C \|g\|_{H^{-1}(T)} \|g\|_{L^2},\qquad \forall g\in L^2(0,1),
$$
for some $C>0$. Since 
$$
\mu_1(g)=\mu_0(g)-(Pg)(1),
$$
and because
$$
|\mu_0(g)|\leq \|g\|_{H^{-1}(T)}\lesssim \|g\|_{L^2},
$$
we can conclude that~\eqref{mu1weakcontinuous} holds.
\end{proof}

In view of Lemma~\ref{lemma2} we see that
\begin{eqnarray}\label{prodh1}
(f,g)\mapsto (Pf|Pg)_{L^2}+\mu_0(f)\mu_0( \bar g),
\end{eqnarray}
defines an inner product in $H^{-1}(T)$ whose associated norm is equivalent to the standard norm of $H^{-1}(T)$. We will stick to this inner product on $H^{-1}(T)$ throughout this article, and in particular we denote
$$\|f\|^2_{H^{-1}(T)}:=\|Pf\|^2_{L^2}+\mu_0(f)\mu_0( \bar f).$$

We will repeatedly make use of the following integration-by-parts-type formula.
\begin{lemma}\label{intbp}
Let $u\in H^1(0,1)$, $c\in \mathbb C$ and $h\in L^2(0,1)$. Then
$$(\Id _m^{-1}(u'')+c\delta_1|h)_{H^{-1}(T)}=\left(\begin{pmatrix}
c+u(1)\\ u(0)-u(1)
\end{pmatrix}\Big| \begin{pmatrix}
\mu_0(h)\\ \mu_1(h)
\end{pmatrix}\right)_{\mathbb C^2}-(u|h)_{L^2}.$$
\end{lemma}

\begin{proof}
Set $g=\Id _m^{-1}(u'')+c\delta_1$, then by Lemma~\ref{lemmasecondderivative} we have
$$
P(g)=u'-a,
$$
where $a:=u(1)-u(0)$. Accordingly,
$$
(Pg|Ph)_{L^2}=\int_0^1 (u'(x)-a) (P\bar h)(x) \,dx=\int_0^1 u'(x)(P\bar h)(x) \,dx,
$$
since $Ph$ has mean zero by Corollary~\ref{rkpf}.
Integration by parts yields
$$
(Pg|Ph)_{L^2}
=
 -\int_0^1 u(x) \overline{(Ph)'}(x)dx+[u (P\bar h)]_0^1.
$$
By Lemma~\ref{lemma1} and Corollary~\ref{rkpf}, we deduce that
$$
(Pg|Ph)_{L^2}
=
 -\int_0^1 u(x) \overline{h(x)}dx+u(1)\mu_0(\overline{h})+(u(0)-u(1))\mu_1(\overline{h}).
$$
This shows that
\begin{eqnarray*}
-(g|h)_{H^{-1}(T)}&=&-(Pg|Ph)_{L^2}-\mu_0(g) \mu_0(\overline{h})\\
&=& \int_0^1 u(x) \overline{h}(x)dx-(c+u(1))\mu_0(\overline{h})+(u(1)-u(0))\mu_1(\overline{h}),
\end{eqnarray*}
as we wanted to prove.
\end{proof}

\begin{rem}
Our Hilbert space $(H,\|\cdot\|_H)$ agrees with the space denoted by $B^1_2$ and termed the \emph{Bouziani space} in~\cite{Bou02b} and some subsequent papers by Bouziani himself and other authors, since by~\eqref{Pmu1}
$$(f|g)_H=\int_0^1 {\mathcal I} f(x) {\mathcal I} \bar g(x) dx,\qquad \forall f,g\in V.$$
\end{rem}

\section{Well-posedness results}\label{sec:homog}

In this section we propose a general Hilbert space setting in order to study both the heat and the wave equations under (generalizations of) the integral constraints
\begin{equation}\label{nlbc}
\int_0^1u(x)\,dx=\int_0^1xu(x)\,dx=0.
\end{equation}
Namely, we take the spaces $V_Y$ equipped with the $L^2$-inner product and $H$ (resp. $H^{-1}(T)$) with the inner product in~\eqref{prodH}. According to Corollary~\ref{cor2.5}, we set
$$
H_Y:=\left\{
\begin{tabular}{ll}
$H$ &if $Y=\{0\}^2$ or $Y=\{0\}\times \mathbb C$,\\
$H^{-1}(T)$ &else.
\end{tabular}
\right.
$$
and therefore $V_Y$ is dense in $H_Y$.

Our discussion is based on the weak formulation of our evolution equations, and in particular on the theory of forms. We recall that, in accordance with the terminology of~\cite{Are04}, a
sesquilinear form $a:V_Y\times V_Y\to \mathbb C$ is called \emph{$H_Y$-elliptic} (or simply \emph{elliptic}) if there exist $\alpha>0$ and $\omega\ge 0$ such that
$$\Real a(f,f)+\omega\Vert f\Vert^2_{H_Y}\geq \alpha \| f\|^2_{V_Y},\qquad \forall f\in V_Y;$$
it is called \emph{coercive} if it is elliptic with $\omega=0$; and finally it is called \emph{accretive} if
$$\Real a(f,f)\geq 0,\qquad \forall f \in V_Y.$$
We also say that it satisfies the \emph{Crouzeix estimate} if for some $M>0$
$$|\Ima a(f,f)|\le M\|f\|_{V_Y} \|f\|_{H_Y},\qquad \forall f\in V_Y.$$
(The name is due to the fact that forms that satisfy the Crouzeix estimate also fit the framework of~\cite{Cro04}).

Coming back to our setting, let $Y$ be an arbitrary but fixed subspace of $\mathbb C^2$ and $K$ be a $2\times 2$-matrix and consider the sesquilinear form $a_K$ defined by
\begin{equation}
\label{definform2}
a_K(f,g):=(f|g)_{L^2}+(K\Gamma_1 f|\Gamma_1 g)_{\mathbb C^2},\qquad f,g\in V_Y,
\end{equation}
i.e.,
$$a_K(f,g)=\int_0^1f(x) \overline{g}(x)\,dx+
 (\mu_0(f) \;\;\mu_1(f)) K^*
\begin{pmatrix}
\mu_0(\overline{g})
\\
\mu_1(\overline{g})
\end{pmatrix},\qquad\forall f,g\in V_Y,$$
with form domain $V_Y$. 
By Lemma~\ref{lemmamu1weakcontinuous} and a standard application of Young's inequality we can easily deduce $H^{-1}(T)$-ellipticity of $a_K$. Furthermore, using Lemma~\ref{lemmamu1weakcontinuous} one sees that $a_K$ satisfies the Crouzeix estimate.

Since $V_Y$ is dense in $H_Y$, {the Lax--Milgram Lemma yields that the form $a_K$ defined on  $V_Y$ is associated with a unique linear operator $(A_{Y,K}, D(A_{Y,K}))$ defined by
\begin{eqnarray*}
D(A_{Y,K})&:=&\left\{f\in V_Y: \exists g\in H_Y: a_K(f,h)=\int_0^1 (P g)(x)(P\bar h)(x)\,dx+\mu_0(g)\mu_0(\bar h)\;\;\forall h\in V_Y\right\},\\
A_{Y,K} f&:=&g.
\end{eqnarray*}
Note that in the case $Y=\{0\}^2$, the second term on the right hand side of~\eqref{definform2} vanishes and hence $(A_{Y,K}, D(A_{Y,K}))$ does not really depend on $K$. This is why in the following we denote it simply by $(A,D(A))$).

\begin{theo}\label{identK3}
If $Y\not=\{0\}^2$ and $Y\not=\{0\}\times \mathbb C$, then one has
\begin{eqnarray*}
D(A_{Y,K})&=&\left\{f\in H^1(0,1):\begin{pmatrix}
\mu_0(f)\\ \mu_1(f)
\end{pmatrix}\in Y \hbox{ and there exists a unique } c(f)\in \mathbb C \hbox{ such that } (\ref{cuniqu}) \hbox{ holds } \right\},\\
A_{Y,K}f&=&-\Id _m^{-1}(f'')-c(f)\delta_1,
\end{eqnarray*}
%with $c(u)\in \mathbb C$ uniquely determined by the condition
\begin{equation}
\label{cuniqu}
K \begin{pmatrix}
\mu_0(f) \\ \mu_1(f)
\end{pmatrix}+ 
\begin{pmatrix}
c(f)+f(1) \\
f(0)-f(1)
\end{pmatrix}\in Y^\perp.
\end{equation}

If $Y=\{0\}^2$ or $Y\not=\{0\}\times \mathbb C$, the same statement holds with $c(f)=0$.
\end{theo}
\begin{proof}
We denote
$${\mathcal K}_Y:=\left\{f\in H^1(0,1):\begin{pmatrix}
\mu_0(f)\\ \mu_1(f)
\end{pmatrix}\in Y \hbox{ and there exists a unique } c(f)\in \mathbb C \hbox{ such that } (\ref{cuniqu}) \hbox{ holds } \right\}.$$
Let us first check the inclusion $D(A_{Y,K})\subset {\mathcal K}_Y$.
Let $f\in D(A_{Y,K})$. Then $f\in V_Y$ and there exists $g\in H^{-1}(T)$ such that
\begin{equation}\label{sn1}
(f|h)_{L^2}+ (K\Gamma_1 f| \Gamma_1 h)_{\mathbb C^2}=\int_0^1 (P g)(x)(P\bar h)(x)\,dx
+\mu_0(g) \mu_0(\bar {h})
,\qquad\forall h\in V_Y.
\end{equation}
Because $g\in H^{-1}(T)$, by Corollary~\ref{rkpf} $P(Pg)\in H^1(T)$ and integrating by parts we obtain
\begin{eqnarray*}
\int_0^1 (P g)(x)(P\bar h)(x)\,dx&=&\int_0^1 (P(P g))'(x)(P\bar h)(x)\,dx\\&=&
-\int_0^1 (P(P g))(x)\overline{h}(x)\,dx+[P(Pg) (P\bar h)]_0^1,\qquad\forall h\in V_Y.
\end{eqnarray*}
Now, the scalar number
$$-(K\Gamma_1 f |\Gamma_1 h)_{\mathbb C^2} +\mu_0( g) \mu_0(\overline{h})+[P(Pg) (P\bar h)]_0^1\in \mathbb C$$ 
is a linear combination of $\mu_0(\bar h)$ and $\mu_1(\bar h)$, hence it can be written in the form
$$c_0 \mu_0(\overline{h})+c_1 \mu_1(\overline{h})=\int_0^1 (c_0+c_1(1-x)) \overline{h(x)}dx,$$
for some $c_0,c_1\in \mathbb C$.
Letting $p(x):=c_0+c_1(1-x)$, we obtain
 that
$$
(f|h)_{L^2}=(-P(Pg)+p|h)_{L^2},\qquad\forall h\in V_Y,
$$
% here we are using the fact that $V_Y$ is dense in $L^2(0,1)$ !
where $p$ is a polynomial of degree $\leq 1$.
Thus, denoting by $\Pi$ the orthogonal projection of $L^2(0,1)$ onto ${\mathbb P}_1({\mathbb R})$, we obtain (by restricting the previous identity to all $h\in V\subset V_Y$)
$$
(I-\Pi)(f+P(Pg)-p)=0,
$$
or equivalently
\begin{equation}\label{sn2}
f=(I-\Pi)(-P(Pg)+p)+\Pi f=-P(Pg)+\Pi (P(Pg)+f).
\end{equation}
This proves that $f$ belongs to $H^1(0,1)$ and (differentiating~\eqref{sn2} twice)
that $g=-f''$ in the distributional sense (i.e. in ${\mathcal D}'(0,1)$). Hence, by~\eqref{leqiso}, there exists $c(f)\in \mathbb C$ such that
\[
A_{Y,K}f=g=-\Id _m^{-1}(f'')-c\delta_1,
\]
and in fact $c(f)=-\mu_0(g)$. Note that if $Y=\{0\}^2$ or $Y\not=\{0\}\times \mathbb C$, then $\mu_0(g)=0$ (as $g\in V_Y$) and hence $c(f)=0$.

It remains to check the condition~\eqref{cuniqu}.
We first notice that, for all $f\in D(A_{Y,K})$,~\eqref{sn2} leads to
$$
f'=-Pg+a,
$$
for some $a\in \mathbb C$. By~\eqref{sn1} we obtain
$$
\int_0^1f(x) \bar h(x)\,dx+(K\Gamma_1 f|\Gamma_1 h)_{\mathbb C^2}=\int_0^1 (-f'(x)+a)(P\bar h)(x)\,dx-c(f) \mu_0(\overline{h}),\qquad\forall h\in V_Y.
$$
As $a\int_0^1 (Ph)(x)\,dx=0$ because $Ph\in \tilde{V}$ by Corollary~\ref{rkpf}, we deduce that
$$
\int_0^1f(x) \overline{h}(x)\,dx+(K\Gamma_1 f|\Gamma_1 h)_{\mathbb C^2}=-\int_0^1 f'(x)(P\bar h)(x)\,dx-c(f)\mu_0(\overline{h}),\qquad\forall h\in V_Y.
$$
By integration by parts in the first term on the right-hand side we obtain (since $(Ph)'=h$ by Lemma~\ref{lemma1})
$$
 (K\Gamma_1 f|\Gamma_1 h)_{\mathbb C^2}=f(0) (P\bar h)(0)-f(1) (P\bar h)(1) -c(f)\mu_0(\overline{h}),\qquad\forall h\in V_Y.
$$
By Corollary~\ref{rkpf} we arrive at
$$(K\Gamma_1 f|\Gamma_1 h)_{\mathbb C^2}= (-c(f)-f(1)) \mu_0(\overline{h})+(f(1)-f(0)) \mu_1(\overline{h}), \quad \forall h\in V_Y.$$
By surjectivity of $\Gamma_1$, cf.\ Remark~\ref{surjmu}, we have shown~\eqref{cuniqu}.

Before going on let us notice that if $Y\not=\{0\}^2$ and $Y\not=\{0\}\times \mathbb C$, then \eqref{cuniqu} determines $c(f)$ uniquely. Indeed if $Y=\mathbb C^2$, then \eqref{cuniqu} is equivalent to
$$
K \begin{pmatrix}
\mu_0(f) \\ \mu_1(f)
\end{pmatrix}+ 
\begin{pmatrix}
c(f) +f(1) \\
f(0)-f(1)
\end{pmatrix}=0,
$$
while in the case $Y=$ Span $(1,\alpha)^\top$ with $\alpha\in \mathbb C$,
then
\eqref{cuniqu} is equivalent to 
$$
\left(K \begin{pmatrix}
\mu_0(f) \\ \mu_1(f)
\end{pmatrix}+ 
\begin{pmatrix}
c(f) +f(1) \\
f(0)-f(1)
\end{pmatrix}\Big| \begin{pmatrix}
1\\ \alpha
\end{pmatrix}\right)_{\mathbb C^2}=0,
$$
which again determines $c(f)$ uniquely.

Let us now prove the converse inclusion.
Let then $f\in {\mathcal K}_Y$. Then we can take 
$g=-\Id _m^{-1}(f'')-c(f)\delta_1,
$
with $c(f)\in \mathbb C$ fixed by the condition \eqref{cuniqu}.
 Moreover by 
Lemma~\ref{intbp}, for all $h\in L^2(0,1)$ we may write 
$$(g|h)_{H^{-1}(T)}=-\left(\begin{pmatrix}
c(f) +f(1)\\ f(0)-f(1)
\end{pmatrix}\Big| \begin{pmatrix}
\mu_0(h)\\ \mu_1(h)
\end{pmatrix}\right)_{\mathbb C^2}+(f|h)_{L^2}.$$
But taking $h\in V_Y$ and using \eqref{cuniqu} allow us to transform the first term of this right-hand side
and to obtain
$$(g|h)_{H^{-1}(T)}=\left(K\begin{pmatrix}
\mu_0(f)\\ \mu_1(f)
\end{pmatrix}\Big| \begin{pmatrix}
\mu_0(h)\\ \mu_1(h)
\end{pmatrix}\right)_{\mathbb C^2}+(f|h)_{L^2},\qquad \forall h\in V_Y.$$
This shows that
$$
a_K(f,h)=(g|h)_{H^{-1}(T)} ,\qquad \forall h\in V_Y,
$$
and proves that $f$ belongs to $D(A_{Y,K})$.
\end{proof}

\begin{rem}
Let $q\in C^1([0,1];{\mathbb C})$ with $0<q_0 \le \Real q(x)\le Q_0$ for some $q_0,Q_0\in \mathbb R$ and all $x\in [0,1]$. Then one can consider the form defined by
$$a_K(f,g):=(qf|g)_{L^2}+(K\Gamma_1 f |\Gamma_1 g )_{\mathbb C^2},\qquad f,g\in V_Y.$$
Mimicking the proof of Theorem~\ref{identK3}, one can prove that the associated operator is given by
\begin{eqnarray*}
D(A_{Y,K})&=&\left\{u\in H^1(0,1):\Gamma_1 u \in Y\right\},\\
A_{Y,K}u&=&-\Id _m^{-1}(qu)''-c_q(u) \delta_1,
\end{eqnarray*}
with $c(u)=0$ if $Y=\{0\}^2$ or $Y\not=\{0\}\times \mathbb C$, otherwise $c_q(u)\in \mathbb C$ is uniquely determined by the condition
$$K\begin{pmatrix}
\mu_0(u)\\ \mu_1(u)
\end{pmatrix}+
\begin{pmatrix}
c_q + q(1)u(1)\\ \mu_0((qu)')
\end{pmatrix}
\in Y^\perp.$$
Similar conclusions hold in the case of Theorem~\ref{identK4} below.
We omit the straightforward details.
\end{rem}

\begin{rem}
The operator $(A,D(A))$ associated with $(a_0,V)$ is given by
\begin{eqnarray*}
D(A)&=&\{u\in H^1(0,1):\mu_0(u)=\mu_1(u)=0\}\\
Au&=&-\Id _m^{-1}(u'').
\end{eqnarray*}
This shows that in particular
$$
A u=-u'' \quad\hbox{ in } {\mathcal D}'(0,1).
$$
\end{rem}

\begin{rem}
Contrary to the intuition, for $u\in D(A)\subset H^1(0,1)$ the vector $Au$ may not agree with $-u''$
even if $u''$ belongs to $H^{-1}(T)$. Indeed take the function $u$ defined by
$$
u(x)=|x-\frac12|+\alpha x+\beta,
$$
with $\alpha,\beta\in \mathbb R$ fixed such that
$\mu_0(u)=\mu_1(u)=0$. Hence we easily check that
$$
- u''=-2(\delta_{\frac12}-\delta_1)+\alpha \delta_1' \hbox{ in } {\mathcal D}'(T).
$$
The distribution $-u''$ cannot agree with $Au$ since 
$$
Au=-\Id _m^{-1}(u'')=-2(\delta_{\frac12}-\delta_1),
$$
and elementary calculations confirm that
$$
a(u,h)=\int_0^1 u(x) h(x)\,dx=\int_0^1 P(Au)(x) Ph(x)\,dx, \quad \forall h\in V.
$$
\end{rem}

\begin{cor}\label{prep1}
The operator $A_{Y,K}$ generates an analytic semigroup $(e^{-tA_{Y,K}})_{t\ge 0}$ of angle $\frac{\pi}{2}$ on $H_Y$ and also a cosine operator function with phase space $V_Y\times H_Y$. It is contractive (resp., exponentially stable) if both eigenvalues of $K$ have positive (resp., strictly positive) real part.
Finally $A_{Y,K}$ is self-adjoint and semi-bounded if $K$ is hermitian.
\end{cor} 
Because the embedding of $H^1(0,1)$ in $L^2(0,1)$ is a Hilbert--Schmidt operator (cf.~\cite[Satz~4]{Mau61}) and each $e^{-tA_{Y,K}}$ maps $L^2(0,1)$ into $H^1(0,1)$, each such operator is Hilbert--Schmidt. Recall that the composition of Hilbert--Schmidt operators is of trace class. Due to the semigroup law we hence conclude that $e^{-2tA_{Y,K}}$} is of trace class for all $t>0$: One says that $\left(e^{-tA_{Y,K}}\right)_{t\ge 0}$ is \emph{immediately of trace class}, or sometimes that it is a \emph{Gibbs semigroup}.
\begin{proof}
The form $a_K$ is bounded, $H_Y$-elliptic, and satisfies the Crouzeix estimate. It is coercive (resp., accretive) if both eigenvalues of $K$ have strictly positive (resp., positive) real part. Hence, 
it follows directly from~\cite[Thm.~5]{Cro04} that $A_{Y,K}$ generates a cosine operator function with associated phase space $V_Y\times H_Y$ and  hence an analytic semigroup of angle $\frac{\pi}{2}$ by~\cite[Thm.~3.14.17]{AreBatHie01}. 
\end{proof}

It is known that if a bounded elliptic form is symmetric, then the associated operator is similar to its own part in the form domain, cf.~\cite[\S~5.5.2]{Are04}. Furthermore, the operator associated with a form generates \emph{three} semigroups, cf.~\cite[Chapter 1]{Ouh05}: one on the given Hilbert space, one on the form domain and one on the dual of the form domain (using the given Hilbert space as the pivot space). 
Hence
we obtain the following result concerning well-posedness in a more standard $L^2$-context.

\begin{theo}\label{identVAK}
Let $K$ be hermitian and $Y$ be a subspace of $\mathbb C^2$. Then the semigroup $(e^{-tA_{Y,K}})_{t\ge 0}$ on $H_Y$ leaves $V_Y$ invariant and its restriction is a semigroup on $V_Y$ that is analytic of angle $\frac{\pi}{2}$ and immediately of trace class. Its generator is the part $A_{Y,K}^{V_Y}$ of $A_{Y,K}$ in $V_Y$, which is explicitly given by
\begin{eqnarray*}
D(A_{Y,K}^{V_Y})&=&\left\{u\in D(A_{Y,K})\cap H^2(0,1): 
K\Gamma_1(u)+\Gamma_2(u)\in Y^\perp,\; \Gamma_1(u'')\in Y
\right\},\\
A_{Y,K}^{V_Y}u&=&-u''.
\end{eqnarray*}
\end{theo}

\begin{proof}
The part of $A_{Y,K}$ in $V_Y$ has domain
$$D(A^{V_Y}_{Y,K}):=\{u\in D(A_{Y,K}):A_{Y,K} u\in V_Y\}.$$
But according to Theorem~\ref{identK3}, for $f\in D(A_{Y,K})$ we have
$$
A_{Y,K} f=-\Id _m^{-1}(f'')-c\delta_1,
$$
with $c\in \mathbb C$ fixed by the condition \eqref{cuniqu}.
Therefore 
$$f''=-A_{Y,K} f\in {\mathcal D}'(0,1),
$$
and since
the condition $A_{Y,K} f\in V_Y$ means in particular that
$$A_{Y,K} f\in L^2(0,1),
$$
we deduce that $f$ belongs to $H^2(0,1)$.

On the other hand, using \eqref{ajoutserge10} we get
$$
-A_{Y,K} f=f''+(f'(0)-f'(1)+c)\delta_1\in L^2(0,1).
$$
Consequently $f'(0)-f'(1)+c$ must be zero, i.e.,
$$
c=f'(1)-f'(0)=\mu_0(f''),
$$
and 
$$
A_{Y,K} f=-f'',
$$
as an equality of $L^2$-functions.
By~\eqref{cuniqu} we find
$$K\Gamma_1(u)+\Gamma_2(u)\in Y^\perp.$$
This completes the proof.
\end{proof}

\begin{rem}
If in particular $Y=\{0\}^2$, then each $u\in D(A_{Y,K}^{V_Y})$ satisfies $\mu_0(u'')=\mu_1(u'')=0$. A direct computation shows that this is equivalent to the boundary conditions
\[
u'(1)=u'(0)=u(1)-u(0).
\]
One may extend $A_{\{0\}^2,K}^{V}$ to an operator defined on the whole $L^2(0,1)$ by dropping the conditions $\mu_0(u)=\mu_1(u)=0$ and keeping the above boundary conditions. This new operator is perhaps more natural and has been extensively studied in~\cite{BobMug13}.
\end{rem}

We now obtain the following well-posedness result. It should be compared with the main result in~\cite{BouMerBen08}, which our theorem below widely extends -- in fact, both the allowed initial data are more general and the notion of solution is much stronger.

\begin{theo}\label{wellp2}
Let $Y$ be a subspace of $\mathbb C^2$ and let $K$ be a $2\times 2$-matrix.
\begin{enumerate}[(1)]
\item Then the heat equation
\begin{equation}\label{heat1}
\frac{\partial u}{\partial t}(t,x)=\frac{\partial^2 u}{\partial x^2}(t,x),\qquad t> 0,\; x\in (0,1),
\end{equation}
with moment conditions
\begin{equation}\label{heat2}
P_{Y^\perp}\begin{pmatrix}
\mu_0(u(t))\\ \mu_1(u(t))
\end{pmatrix}=0,\qquad t> 0,
\end{equation}
and
\begin{equation}\label{heat3}
P_Y \left(K\begin{pmatrix}
\mu_0(u(t))\\ \mu_1(u(t))
\end{pmatrix}+ \begin{pmatrix}
\mu_0( u'') +u(1)\\ u(0)-u(1)
\end{pmatrix}\right)=0,\qquad t> 0,
\end{equation}
(here $P_Y$ and $P_{Y^\perp}$ denote the orthogonal projections of $\mathbb C^2$ onto $Y$ and $Y^\perp$, respectively) and initial condition
$$u(0,\cdot)=u_0\in H_Y$$
is governed by an analytic semigroup, thus it is well-posed. If additionally $K$ is positive definite, then 
$$\lim_{t\to \infty}\|u(t,\cdot)\|_{H^{-1}(T)}=0$$
uniformly for all initial data.

\item Similarly, the wave equation
$$\frac{\partial^2 u}{\partial t^2}(t,x)=\frac{\partial^2 u}{\partial x^2}(t,x),\qquad t\ge 0,\; x\in (0,1),$$
with moment conditions~\eqref{heat2}-\eqref{heat3} and initial conditions
$$u(0,\cdot)=u_0\in D(A_{Y,K}^{V_Y})$$
along with
$$\frac{\partial u}{\partial t}(0,\cdot)=u_1\in D(A_{Y,K})$$
is governed by a cosine operator function, and in particular it is well-posed.
\end{enumerate}
\end{theo}

In the proof of this theorem we will need the following result.

\begin{lemma}\label{DA2} There holds
$$D(A^2_{Y,K})=\{u\in H^3(0,1): \Gamma_1 u,\Gamma_1 u''\in Y\hbox{ and } K\Gamma_1 u-\Gamma_2 u,K\Gamma_1 u''-\Gamma_2 u''\in Y^\perp\}.$$
\end{lemma}

\begin{proof}
The inclusion ``$\supset$'' is clear. In order to prove that ``$\subset$'' also holds, take 
$u\in D(A^2_{Y,K})$: Then $u\in H^1(0,1)$ is such that $A_{Y,K}u=-\Id ^{-1}_m u''+c(u)\delta_1\in H^1(0,1)$. It clearly suffices to prove that $u''\in H^1(0,1)$. 
Now, 
\[%\Phi(v)&=&
\langle u'',v,\rangle_{H^{-1}(0,1)-H^1_0(0,1)} 
= \langle \Id _m^{-1} u''-c(u)\delta_1,v\rangle_{H-H^1(T)}
= -\langle A_{Y,K}u,v\rangle_{H-H^1(T)}\qquad \forall v \in {\mathcal D}(0,1).
\]

Because by assumption $A_{Y,K}u\in H^1(0,1)$ we deduce that in fact 
\[\langle u'',v,\rangle_{H^{-1}(0,1)-H^1_0(0,1)}= -( A_{Y,K}u|v)_{L^2},\]
hence $u''=-A_{Y,K}u\in H^1(0,1)$ and we conclude that $u\in H^3(0,1)$, as we wanted to prove. 
In particular, 
\[
\Id ^{-1}_m u''-c(u)\delta_1=u''\in H^1(0,1)\]
by Remark~\ref{noidm}, hence necessarily $c(u)=0$.
\end{proof}

\begin{proof}[Proof of Theorem~\ref{wellp2}]
Since the operator $A_{Y,K}$ is a semigroup generator, well-posedness of the corresponding parabolic problem is clear. By construction this semigroup yields the solution of the evolution equation
\begin{equation}\label{heatidm}
\frac{\partial u}{\partial t}(t,x)={\rm Id\, }_m^{-1}\frac{\partial^2 u}{\partial x^2}(t,x)-c(u(t,x)) \delta_1,\qquad t\ge 0,\; x\in (0,1),
\end{equation}
rather than the standard heat equation. However, by standard semigroup theory we know that $u(t):=e^{-tA_{Y,K}}u_0$ lies in $D(A_{Y,K}^2)$ for any $t>0$, as these semigroups are analytic. By Lemma~\ref{DA2} $D(A_{Y,K}^2)\subset H^2(0,1)$ and the claim follows by Remark~\ref{noidm}.

For the second assertion, since cosine functions keep the regularity of initial data but offer no additional smoothing effect, we are forced to reach $H^2(0,1)$-regularity of solutions (which by Remark~\ref{noidm} is sufficient to finally drop $\Id _m^{-1}$) by actually imposing more regular data. This is obtained if $u_0$ lies in the domain of the generator's part in the form domain; and $u_1$ lies in the generator's domain, respectively.
\end{proof}

A major feature of our semigroup approach -- in particular in comparison with the Galerkin method used in some earlier articles on this subject -- lies in the possibility to deduce optimal regularity results for solutions. For example it is known that due to analyticity the semigroup operators $e^{-tA_{Y,K}}$ map $H_Y$ into $D(A_{Y,K})$ for all $Y$ and $K$ for all $t>0$, hence in particular they map (a closed subspace of) $L^1(0,1)$ into (a closed subspace of) $L^\infty(0,1)$, and by the Dunford--Pettis theorem (see~\cite[\S~7.3.1]{Are04}) we deduce that each $e^{-tA_{Y,K}}$ is an integral operator associated with an $L^\infty$-kernel. Furthermore, the following holds.

\begin{cor}\label{regulsol}
Let $Y$ be a subspace of ${\mathbb C}^2$ and $K$ a $2\times 2$-matrix and let $u_0\in H_Y$. Then the unique solution $u$ to the initial value problem associated with~\eqref{heat1}-\eqref{heat2}-\eqref{heat3} satisfies $u(t,\cdot)\in C^\infty([0,1])$ and moreover $u$ and its derivatives fulfill the non-local boundary conditions
$$\begin{pmatrix}
u^{(2h-1)}(t,1)-u^{(2h-1)}(t,0)\\ u^{(2h-2)}(t,1)-u^{(2h-2)}(t,0)-u^{(2h-1)}(t,0)
\end{pmatrix}\in Y\qquad \hbox{for all }h\in \mathbb N^*$$
along with
$$K\begin{pmatrix}
u^{(2h-1)}(t,1)-u^{(2h-1)}(t,0)\\ u^{(2h-2)}(t,1)-u^{(2h-2)}(t,0)-u^{(2h-1)}(t,0)
\end{pmatrix}-\begin{pmatrix}
u^{(2h+1)}(t,1)-u^{(2h+1)}(t,0)+u^{(2h)}(t,1)\\ u^{(2h)}(t,0)-u^{(2h)}(t,1)
\end{pmatrix}\in Y^\perp\quad \hbox{for all }h\in \mathbb N^*,$$
for all $t>0$.
\end{cor}

For example, for $Y=\{0\}\times \mathbb C$ and $K=0$ this amounts to saying that $u(t,\cdot)$ and all its derivatives fulfill periodic boundary conditions.

\begin{proof}
In a way similar to Lemma~\ref{DA2} it can be proved by induction that in fact for all $h\in \mathbb N$
$$D(A^h_{Y,K})=\{u\in H^{2h+1}(0,1): \Gamma_1 u^{(2\ell)}\in Y\hbox{ and } K\Gamma_1 u^{(2\ell)}-\Gamma_2 u^{(2\ell)}\in Y^\perp \;\forall \ell\le h \},$$
hence in particular
$$\bigcap_{h\in \mathbb N} D(A^h_{Y,K})=\{u\in C^\infty([0,1]): \Gamma_1 u^{(2\ell)}\in Y\hbox{ and } K\Gamma_1 u^{(2\ell)}-\Gamma_2 u^{(2\ell)}\in Y^\perp \;\forall \ell\in \mathbb N \}.$$
Observe that for $u\in H^3(0,1)$
$$\Gamma_2 u=\begin{pmatrix}
u'(1)-u'(0)+u(1)\\ u(0)-u(1)
\end{pmatrix}$$
and moreover
$$\Gamma_1 u''=\begin{pmatrix}
u'(1)-u'(0)\\ u(1)-u(0)-u'(0)
\end{pmatrix}$$
along with
$$\begin{array}{rcll}
\Gamma_2 u''&=&\begin{pmatrix}
\mu_0 (u^{(4)})+u''(1)\\ u''(0)-u''(1)
\end{pmatrix}\qquad &\hbox{if }u\in H^3(0,1)\\
&=&\begin{pmatrix}
u^{(3)}(1)-u^{(3)}(0)+u''(1)\\ u''(0)-u''(1)
\end{pmatrix}\qquad & \hbox{if }u\in H^4(0,1).
\end{array}$$
Reasoning similarly we can prove by induction that in fact for $u\in C^\infty([0,1])$
$$\Gamma_1 u^{(2h)}=\begin{pmatrix}
u^{(2h-1)}(1)-u^{(2h-1)}(0)\\ u^{(2h-2)}(1)-u^{(2h-2)}(0)-u^{(2h-1)}(0)
\end{pmatrix},\qquad 
\Gamma_2 u^{(2h)}=\begin{pmatrix}
u^{(2h+1)}(1)-u^{(2h+1)}(0)+u^{(2h)}(1)\\ u^{(2h)}(0)-u^{(2h)}(1)
\end{pmatrix},$$
for all $h\in \mathbb N^*$. This concludes the proof, since it is well-known that an analytic semigroup maps immediately into the domain of any power of its generator.
\end{proof}

\section{Spectral analysis}\label{sec:spectral}

Reminding that $A_{Y,K}$ has compact resolvent (due to the compact embedding of $L^2(0,1)$ into $H^{-1}(T)$) and that $A_{Y,K}$ is self-adjoint if $K$ is hermitian, we promptly obtain the following.

\begin{lemma}\label{lemma:specgener}
Let $Y$ be a subspace of $\mathbb C^2$ and let $K$ be a $2\times 2$-matrix. Then the operator $A_{Y,K}$ has pure point spectrum, which lies in $\mathbb R$ if $K$ is hermitian.
\end{lemma}

In view of this Lemma, if $K$ is hermitian we denote by 
$\lambda_{Y,K,k}^2$, $k\in {\mathbb N}^*$, the eigenvalues of $A_{Y,K}$ enumerated in increasing order.

In this section we will describe the spectrum of the operator $A_{Y,K}$
for all possible subspaces $Y$ and when $K$ is hermitian, obtaining in particular in all cases an asymptotic result of Weyl's type. While we do not discuss the dependence of the spectrum with respect to the variation of the subspaces $Y$, the spaces $Y=\{0\}^2$ and $Y={\mathbb C}^2$ represent in fact the extremal cases, in the following sense.

\begin{prop}\label{prop:specgener}
Let $Y_1,Y_2$ be subspaces of $\mathbb C^2$ and let $K_1,K_2$ be hermitian $2\times 2$-matrices. Denote by $A_{Y_1,K_1}$ and $A_{Y_2,K_2}$ the operators associated with the form $a_{K_1}$ with domain $V_{Y_1}$ and with the form $a_{K_2}$ with domain $V_{Y_2}$, respectively. 
%Denote by $\nu_k(A_{Y_1,K_1})$ (resp., $\nu_k(A_{Y_2,K_2})$) the $k^{\rm th}$-eigenvalue of %$A_{Y_1,K_1}$ (resp., $A_{Y_2,K_2}$). \\
If $Y_2$ is a subspace of $Y_1$ and the matrix $K_2-K_1$ is positive semidefinite, then 
$$\lambda_{Y_1,K_1,k}^2\le \lambda_{Y_2,K_2,k}^2,\qquad \forall k\in \mathbb N.$$
\end{prop}
\begin{proof}
The assertion is a direct consequence of the Courant--Fischer minimax theorem, since the operators $A_{Y_1,K_1},A_{Y_2,K_2}$ are self-adjoint on Hilbert spaces ($H_{Y_1}$ and $H_{Y_2}$ respectively) that are endowed with the same norm and moreover $V_{Y_2}$ is a subspace of $V_{Y_1}$ under the above assumptions.
\end{proof}
Combining the previous results we can slightly improve the assertion on exponential stability of $(e^{-tA_{Y,K}})_{t\ge 0}$ contained in Corollary~\ref{prep1}, under the assumption that $K$ is just positive semidefinite.

\begin{rem}\label{rk4.3}
Here and in the remainder of this section we repeatedly use the fact that eigenfunctions of $A_{Y,K}$ are in $D(A^2_{Y,K})$ by a standard bootstrap argument, and hence in $H^2(0,1)$ by Lemma~\ref{DA2}. By Remark~\ref{noidm} we can hence regard them as solutions of the more usual eigenvalue problem
$$-u''=\lambda^2_{Y,K,k} u.$$ 
\end{rem}

\begin{cor}\label{cor:specgener}
Let $Y$ be any subspace of $\mathbb C^2$ and $K$ be hermitian and positive semidefinite. Then each eigenvalue of $A_{Y,K}$ is strictly positive, and in particular the semigroup $(e^{-tA_{Y,K}})_{t\ge 0}$ is uniformly exponentially stable.
%% Moreover, let $K=0$: Then $0$ lies outside the spectrum, unless $Y$ is the 1-dimensional space spanned by 
%% $$\left(2,1+\frac{i}{\sqrt{3}} \right)\quad \hbox{or}\quad \left(2,1-\frac{i}{\sqrt{3}} \right).$$
\end{cor}
\begin{proof}
Combining Lemma~\ref{lemma:specgener} and Proposition~\ref{prop:specgener} we deduce that for any 1-dimensional subspace $Y$ of $\mathbb C^2$ the $k^{\rm th}$ eigenvalue of $A_{Y,K}$ is always contained in the interval
$$[\lambda_{\mathbb C^2,K,k}^2,\lambda_{\{0\}^2,K,k}^2]=[\lambda_{\mathbb C^2,K,k}^2,\lambda_{\{0\}^2,0,k}^2]$$
(In view of Theorem~\ref{identK3}, $A_{\{0\}^2,K}=A_{\{0\}^2,0}$).
In fact, letting $K$ be positive semidefinite we deduce -- again from Proposition~\ref{prop:specgener} -- that 
$$\lambda_{\mathbb C^2,0,k}^2\le \lambda_{\mathbb C^2,K,k}^2.$$
Hence, it suffices to show that all eigenvalues of $A_{\mathbb C^2,0}$ are strictly positive. First of all, $A_{\mathbb C^2,0}$ is accretive, hence all its eigenvalues are positive. To show that 0 is not an eigenvalue, and hence that $\lambda_{\mathbb C^2,0,1}^2>0$, take $u$ such that $A_{\mathbb C^2,0}u=u''=0$, i.e., 
$$u(x):=ax+b,\qquad x\in (0,1),$$ 
for some $a,b\in \mathbb C$. Observe that
%$$\mu_0(u)=\frac{a}{2}+b\quad \hbox{and}\quad \mu_1(u)=\frac{a}{6}+\frac{b}{2},$$
%while
$$\mu_0( u'') +u(1)=a+b\quad \hbox{and}\quad u(0)-u(1)=-a.$$
If we impose that 
%$(\mu_0(u),\mu_1(u))=(0,0)$ or 
 $(\mu_0( u'') +u(1),u(0)-u(1))=(0,0)$, then clearly $a=b=0$, i.e., $u\equiv 0$.
This concludes the proof.
\end{proof}

%\subsection{The case of operators acting on zero mean functions}

Now we look for the eigenvalues and the eigenvectors of the operator $A_{Y,K}$ in the hermitian case and deduce a formula of Weyl's type. Hence we can always assume that an eigenvalue $\lambda$ of $A_{Y,K}$ is real. Let $\lambda\in \mathbb R$ and $u\in D(A_{Y,K})$ be such that
$$
A_{Y,K}u=\lambda^2 u.
$$
Remark~\ref{rk4.3} yields
$$u''=-\lambda^2 u.$$
By Corollary~\ref{cor:specgener} $\lambda\not=0$, 
%Since $A$ is a positive self-adjoint operator, $\lambda$ is different from zero,
and therefore $u$ is of the form
\begin{equation*}\label{eigenvector}
u(x)=c_1 \cos(\lambda x)+c_2 \sin(\lambda x),
\end{equation*}
for some real numbers $c_1$ and $c_2$. By direct computations we see that
\begin{eqnarray*}\label{momentseigen1}
\mu_0(u)=\int_0^1u(x)\,dx=\frac{1}{\lambda}(c_1 \sin \lambda+(1 - \cos \lambda)c_2),
\end{eqnarray*}
and
\begin{eqnarray*}
\mu_1(u)=\int_0^1(1-x)u(x)\,dx=\frac{1}{\lambda^2}\left((1 - \cos \lambda) c_1+(\lambda-\sin \lambda) c_2\right).
\label{momentseigen2}
\end{eqnarray*}
For the sake of later reference we also observe that
\begin{eqnarray*}
u(1)&=&c_1 \cos\lambda +c_2 \sin \lambda,\\
u(0)&=&c_1,\\
u'(1)&=&-\lambda c_1 \sin \lambda+\lambda c_2 \cos \lambda,\\
u'(0)&=&\lambda c_2,
\end{eqnarray*}
so that 
\begin{equation}\label{gamma1c1c2}
\Gamma_1 u ={\mathcal M}(\lambda) {\mathcal B}(\lambda)\begin{pmatrix}
c_1 \\ c_2
\end{pmatrix}:=\begin{pmatrix}
\frac{1}{\lambda} & 0 \\ 0 & \frac{1}{\lambda^2}
\end{pmatrix} \begin{pmatrix}
\sin \lambda & 1-\cos \lambda\\
1-\cos \lambda &\lambda -\sin \lambda 
\end{pmatrix}\begin{pmatrix}
c_1 \\ c_2
\end{pmatrix}
\end{equation}
while
\begin{equation}\label{gamma2c1c2}
\Gamma_2 u ={\mathcal C}(\lambda)\begin{pmatrix}
c_1 \\ c_2
\end{pmatrix}:=\begin{pmatrix}
\lambda \sin \lambda-\cos \lambda &\lambda (1-\cos \lambda)-\sin \lambda\\
\cos \lambda - 1 & \sin \lambda 
\end{pmatrix}\begin{pmatrix}
c_1 \\ c_2
\end{pmatrix},
\end{equation}
where $\Gamma_1,\Gamma_2$ are the operators introduced in~\eqref{gammagrand}.

If we are imposing integral conditions associated with a general subspace $Y$ and a general hermitian matrix $K$, then by Theorem~\ref{identK3} we know that the relevant conditions are
$$P_{Y^\perp}\Gamma_1 u=0 \qquad \hbox{and}\qquad P_Y (K\Gamma_1 u-\Gamma_2 u)=0 ,$$
or rather, taking into account~\eqref{gamma1c1c2} and~\eqref{gamma2c1c2}, 
$$\begin{pmatrix}
P_{Y^\perp}{\mathcal M}(\lambda){\mathcal B}(\lambda)\\
P_Y (K{\mathcal M}(\lambda){\mathcal B}(\lambda)-{\mathcal C}(\lambda))
\end{pmatrix}
\begin{pmatrix}
c_1\\ c_2
\end{pmatrix}
=0.$$
Since an eigenvalue corresponds to a non-trivial solution $(c_1,c_2)$ of this linear system, we directly obtain the following result.
\begin{theo}\label{tspectralDelio}
Let $K$ be a hermitian $2\times 2$-matrix and $Y$ be a subspace of $\mathbb C^2$. A number $\lambda\in \mathbb R^*$ is such that $\lambda^2$ is an eigenvalue of $A_{Y,K}$ if and only if the $4\times 2$-matrix
$$\begin{pmatrix}
P_{Y^\perp}{\mathcal M}(\lambda){\mathcal B}(\lambda)\\
P_Y (K{\mathcal M}(\lambda){\mathcal B}(\lambda)-{\mathcal C}(\lambda))
\end{pmatrix}$$
is of  rank 0 or 1.
\end{theo}

\begin{rem}\label{rankcond-gen} 
%Of course, in our setting the space $Y$ can only have dimension 0, 1 or 2. 
The condition in Theorem~\ref{tspectralDelio} can be specialised in an easy way. 
\begin{enumerate}
\item If $Y=\{0\}^2$, then the condition in Theorem~\ref{tspectralDelio} reduces to
\begin{equation}\label{careq}
D(\lambda)=\lambda \sin\lambda+2\cos\lambda-2=0.
\end{equation}
In this case, each eigenvalue $\lambda$ is simple and its associated eigenspace is spanned by the function
$$
x\mapsto (\sin \lambda-\lambda) \cos(\lambda x)+ (1-\cos \lambda) \sin(\lambda x).
$$
\item In the special case of $Y={\mathbb C}^2$, the condition in Theorem~\ref{tspectralDelio} simplifies to requiring that $\lambda$ be a root of the equation
$$D_K(\lambda):={\rm det}\; (K{\mathcal M}(\lambda){\mathcal B}(\lambda)-{\mathcal C}(\lambda))=0.$$

\item Similarly, if $Y$ is spanned by $(x,y)$ with $x,y\in \mathbb C$ such that $|x|^2+|y|^2=1$, then
$$
P_Y\begin{pmatrix}u\\v\end{pmatrix}= (\bar x\;\; \bar y)\cdot \begin{pmatrix}
 u\\ v
\end{pmatrix}
\begin{pmatrix}x\\y\end{pmatrix}\quad
\hbox{ and }
P_{Y^\perp}\begin{pmatrix}u\\v\end{pmatrix}= (-\bar y\;\; \bar x)\cdot \begin{pmatrix}
u\\ v
\end{pmatrix}
\begin{pmatrix}
-y \\ x
\end{pmatrix}.
$$
Consequently, 
$$P_Y\begin{pmatrix}u\\v\end{pmatrix}=0\quad\hbox{ (resp. }P_{Y^\perp}\begin{pmatrix}u\\v\end{pmatrix}=0\hbox{)}$$ 
if and only if
$$(\bar x\;\; \bar y)\cdot 
\begin{pmatrix}
 u\\
 v
\end{pmatrix}=0\qquad\hbox{ (resp. }(-\bar y\;\; \bar x)\cdot 
\begin{pmatrix}
 u\\ v
\end{pmatrix}
=0\hbox{).}$$ 
This characterization and Theorem~\ref{tspectralDelio} show that $\lambda^2$ is an eigenvalue of $A_{Y,K}$ if and only if 
\begin{equation}
\label{Dalpha}
D_{Y,K}(\lambda)={\rm det}\;
\begin{pmatrix}
(-\bar y, \bar x){\mathcal M}(\lambda){\mathcal B}(\lambda)\\
(\bar x, \bar y) (K{\mathcal M}(\lambda){\mathcal B}(\lambda)-{\mathcal C}(\lambda))
\end{pmatrix}=0.
\end{equation}
\end{enumerate}
\end{rem}

\begin{cor}\label{cWeylgenY}
Let $K$ be hermitian and $Y$ be a subspace of $\mathbb C^2$. We distinguish the following cases.
\begin{enumerate}[(i)]
\item If $Y=\{0\}^2$, then 
\begin{equation}\label{asbehavioreven}
\lambda_{\{0\}^2,K,2k}=2k\pi,\qquad\forall k\in {\mathbb N}^*,
\end{equation}
while
\begin{equation}\label{asbehaviorodd}
\lambda_{\{0\}^2, K,2k-1}=(2k-1)\pi+O(k^{-1}),\qquad\forall k\in {\mathbb N}^*.
\end{equation}
\item If $Y=\mathbb C^2$, then
\begin{equation}\label{asbehaviora1}
\lambda_{\mathbb C^2,K,2k}=2k\pi+O({k^{-1}})\quad \hbox{ and } \quad \lambda_{\mathbb C^2,K,2k-1}=2k\pi+O(k^{-1}),\qquad \forall k\in {\mathbb N}^*.
\end{equation}
\item If $Y$ is spanned by $(0,1)$, then 
\begin{equation}\label{asbehavioraY=01}
\lambda_{Y,K,2k}=2k\pi+O({k^{-1}})\quad \hbox{ and } \quad \lambda_{Y,K,2k-1}=2k\pi+O(k^{-1}),\qquad \forall k\in {\mathbb N}^*.
\end{equation}
\item If $Y$ is spanned by $(1,\alpha)$ with $\alpha\in \mathbb C$, then 
\begin{equation}\label{asbehaviora2}
\lambda_{Y,K,k}=k\pi+O({k^{-1}}),\qquad \forall k\in {\mathbb N}^*.
\end{equation}
\end{enumerate}
In all cases, the Weyl-type asymptotics
$$\lim_{k\to \infty}\frac{\lambda^2_{Y,K,k}}{k^2\pi^2}=1$$
holds
\end{cor}
\begin{proof}
We prove (i) by directly checking that $2k\pi$ is a solution of~\eqref{careq} for all $k\in {\mathbb N}^*$.
For the other roots, we notice that~\eqref{careq} is then equivalent to
$$
f_{\infty}(\lambda)+r(\lambda)=0,$$
where
$ f_{\infty}(z):=\sin z$ is analytic and the remainder $r(z):=\frac{2(\cos z-1)}{z}$ is also analytic (except at $z=0$) and satisfies
$$|r(z)|\leq \dfrac{4}{|z|},\qquad \forall |z|\ne 0.
$$
We prove~\eqref{asbehaviorodd} by applying Rouch\'e's theorem in the ball $B_k=B_k(k \pi,\epsilon_k)$ where $0<\epsilon_k \leq 1 $ will be fixed later on.
We first estimate $|f_{\infty}(z)|$ from below on $\partial B_k$. Suitable trigonometric formulae yield
$$
 %\begin{array}{lll}
|f_{\infty}(k \pi+ \epsilon_k e^{i t})|^2= |f_{\infty}(\epsilon_k e^{i t})|^2 
=\cosh^2 (\epsilon_k \sin t)- \cos^2(\epsilon_k \cos t).
%\end{array} 
$$
Hence we have
$$|f_{\infty}(k \pi+ \epsilon_k e^{i t})|^2\geq \cosh (\epsilon_k \sin t)- \cos(\epsilon_k \cos t).$$
Using the inequalities 
$$\cosh(x)\geq 1+\dfrac{x^2}{2}\quad\hbox{and}\quad \cos(x)\leq 1-\dfrac{x^2}{2}+\dfrac{x^4}{24},\qquad \forall x\in [0,1],
$$ 
we obtain
$$|f_{\infty}(2 k \pi+ \epsilon_k e^{i t})|^2\geq 
\dfrac{1}{2}\epsilon_k^2-\dfrac{1}{24}\epsilon_k^4 \cos^4 t \geq 
\dfrac{1}{2}\epsilon_k^2-\dfrac{1}{24}\epsilon_k^4\geq \dfrac{11}{24}\epsilon_k^2.
$$
On the other hand
$$|r(k \pi+ \epsilon_k e^{i t})| \leq \dfrac{4}{k\pi- 1}.$$ 
Hence if we take $\epsilon_k={\dfrac{C}{k}},$ with $C$ chosen large enough that 
$$\sqrt{\dfrac{11}{24}}\epsilon_k > \dfrac{4}{k\pi- 1},\qquad\forall k\geq 1,$$
 we have 
$$|r(z)|<|f_{\infty}(z)|, \qquad \forall z\in \partial B_k.$$ 

According to Rouch\'e's theorem, for $k$ large enough $f$ admits a unique root in the ball $B_k$, since $f_{\infty}$
has this property.
This proves~\eqref{asbehaviorodd}.

In order to prove (ii) we closely follow the arguments used to show (i).
By direct calculations we see that
$$
D_{K}(\lambda)=2\lambda (f_{\infty}(\lambda)+r(\lambda)),
$$
where 
$f_{\infty}(z):=1-\cos z$ is analytic and the remainder $r(z):=g_1(z) z^{-1}+\ldots+g_4(z) z^{-4}$ (the $g_i$ being finite linear combinations of $1$, $\cos \lambda$, $\sin \lambda$,
$\cos^2 \lambda$ and $\cos \lambda\sin \lambda$) 
is also analytic (except at $z=0$). The conclusion then follows by applying Rouch\'e's theorem in the ball $B_k=B_k(2k \pi,\epsilon_k)$ where $0<\epsilon_k \leq 1 $ is fixed appropriately.

Finally, (iii) and (iv) follow because in view of~\eqref{gamma1c1c2} and~\eqref{gamma2c1c2}, we see that
\begin{eqnarray*}
(-\bar y\;\; \bar x){\mathcal M}(\lambda) {\mathcal B}(\lambda)&=& \frac{1}{\lambda^2} (-\bar y\;\; \bar x)\begin{pmatrix}
{\lambda} \sin \lambda & {\lambda} (1-\cos \lambda)\\
1-\cos \lambda & \lambda -\sin \lambda
\end{pmatrix}\\
&=&\frac{1}{\lambda^2}
(- {\bar y}{\lambda} \sin \lambda + {\bar x} (1-\cos \lambda)\; \;\; - {\bar y} \lambda (1-\cos \lambda) + {\bar x} (\lambda -\sin \lambda) )
\\
&=&\frac{1}{\lambda^2}
(- {\bar y}{\lambda} \sin \lambda + {\bar x} (1-\cos \lambda)\; \;\; \lambda ({\bar x}- {\bar y}(1-\cos \lambda)) - {\bar x} \sin \lambda )
\end{eqnarray*}
while
\begin{eqnarray*}
(\bar x\;\; \bar y)K{\mathcal M}(\lambda) {\mathcal B}(\lambda)-{\mathcal C}(\lambda)
&=&
(\bar x\;\; \bar y)\left[\begin{pmatrix}
\frac{k_{11}}{\lambda} \sin \lambda +\frac{k_{12}}{\lambda^2}(1-\cos \lambda) & \frac{k_{11}}{\lambda} (1-\cos \lambda)+\frac{k_{12}}{\lambda^2}(\lambda -\sin \lambda)\\
\frac{k_{21}}{\lambda} \sin \lambda +\frac{k_{22}}{\lambda^2}(1-\cos \lambda) & \frac{k_{21}}{\lambda} (1-\cos \lambda)+\frac{k_{22}}{\lambda^2}(\lambda -\sin \lambda) 
\end{pmatrix}\right. \\
&&\qquad\qquad\qquad -
\left. \begin{pmatrix}
\lambda \sin \lambda-\cos \lambda &\lambda (1-\cos \lambda)-\sin \lambda\\
\cos \lambda - 1 & \sin \lambda 
\end{pmatrix}\right]
\\
&=&\frac{1}{\lambda^2} (\bar x\;\; \bar y)\left[\begin{pmatrix}
{k_{11}}{\lambda} \sin \lambda +{k_{12}}(1-\cos \lambda) & {k_{11}}{\lambda} (1-\cos \lambda)+{k_{12}}(\lambda -\sin \lambda)\\
{k_{21}}{\lambda} \sin \lambda +{k_{22}}(1-\cos \lambda) & {k_{21}}{\lambda} (1-\cos \lambda)+{k_{22}}(\lambda -\sin \lambda) 
\end{pmatrix}\right. \\
&&\qquad\qquad\qquad -
\left. \begin{pmatrix}
\lambda^3 \sin \lambda-\lambda^2\cos \lambda &\lambda^3 (1-\cos \lambda)-\lambda^2\sin \lambda\\
\lambda^2(\cos \lambda - 1) & \lambda^2\sin \lambda 
\end{pmatrix}\right]
\\
&=&\frac{1}{\lambda^2}
\left(\begin{array}{ll}
- {\bar x}(\lambda^3 \sin \lambda-\lambda^2 \cos \lambda)
+ {\bar y} \lambda^2 (1-\cos \lambda)+q_1(\lambda)\\
 -\bar x\lambda^3 (1-\cos \lambda)
+\bar x\lambda^2 \sin\lambda - {\bar y} \lambda^2 \sin \lambda +q_2(\lambda)
\end{array}
\right)^\top
\\
&=&\frac{1}{\lambda^2}
\left(\begin{array}{ll}
- {\bar x}\lambda^3 \sin \lambda+\lambda^2 ( {\bar x}\cos \lambda
+ {\bar y} (1-\cos \lambda))+q_1(\lambda)\\ -\bar x\lambda^3 (1-\cos \lambda)
+ \lambda^2 \sin\lambda (\bar x- {\bar y}) +q_2(\lambda)
\end{array}
\right)^\top,
\end{eqnarray*}
where $q_i(\lambda)$ is a polynomial in $\lambda$ of degree $\leq 1$ with coefficients which are 
 finite linear combinations of $1$, $\cos \lambda$, $\sin \lambda$, 
$\cos^2 \lambda$.
Therefore by~\eqref{Dalpha} we obtain that
%ici
\begin{eqnarray*}
D_{Y,K}(\lambda)&=&\frac{1}{\lambda^4}{\rm det}\;
\begin{pmatrix}
- {\bar y}{\lambda} \sin \lambda + {\bar x} (1-\cos \lambda)& \lambda ({\bar x}- {\bar y}(1-\cos \lambda)) - {\bar x} \sin \lambda \\
- {\bar x}\lambda^3 \sin \lambda+\lambda^2 ( {\bar x}\cos \lambda
+ {\bar y} (1-\cos \lambda))+q_1(\lambda)& -\bar x\lambda^3 (1-\cos \lambda)
+ \lambda^2 \sin\lambda (\bar x- {\bar y}) +q_2(\lambda))
\end{pmatrix}
\\
&=&\bar x \sin \lambda+2\lambda^{-1}(\bar y^2-\bar x \bar y-\bar x^2) (1-\cos\lambda)+
g_2(\lambda) \lambda^{-2}+\ldots+g_4(\lambda) \lambda^{-4},
\end{eqnarray*}
where the $g_i$ are finite linear combinations of $1$, $\cos \lambda$, $\sin \lambda$, 
$\cos^2 \lambda$, and $\cos \lambda\sin \lambda$.
Hence the arguments in the proof of (ii) yield the conclusion by dividing the cases $x=0$ and $x\not=0$.

In all above cases, an eigenvalue asymptotics of Weyl's type follows directly from~\eqref{asbehavioreven} and~\eqref{asbehaviorodd} (resp. ~\eqref{asbehaviora1}, ~\eqref{asbehavioraY=01}, ~\eqref{asbehaviora2}).
\end{proof}

\begin{rem}\label{rasindK}
In the case $Y=\{0\}^2$, using a Taylor expansion of higher order, we can even show that
$$
\lambda_{\{0\}^2, K,2k-1}=(2k-1)\pi-\dfrac{1}{4 (2k-1) \pi}+O(k^{-2}),\qquad\forall k\in {\mathbb N}^*.
$$
In the same manner, in the case $Y=\mathbb C^2$, we see that the highest order term of the asymptotic of the eigenvalues does not depend on $K$,
but lower order terms do.
Indeed using a Taylor expansion of higher order we can even show that there exist two real numbers $C_1$ and $C_2$
 (depending on $K$) such that
$$
\lambda_{\mathbb C^2,K, 2k-1}=2k \pi+\dfrac{C_1}{k}+O(k^{-2}), \qquad
\lambda_{\mathbb C^2,K, 2k}=2k\pi+\dfrac{C_2}{k}+O(k^{-2}),\qquad\forall k\in {\mathbb N}^*.
$$
\end{rem}

\begin{rem}
The eigenvalue asymptotics of Weyl's type in the third and fourth case of the previous Corollary is also a consequence of the 
comparison principle of Proposition~\ref{prop:specgener} and the Weyl-type formula for the first and second cases.
On the other hand formulas~\eqref{asbehavioraY=01} and~\eqref{asbehaviora2} (that cannot be deduced from the comparison principle) are more precise than the formula of Weyl's type.
As in Remark~\ref{rasindK}, the highest order term of the eigenvalue asymptotics does not depend on $K$,
but lower order terms do.
\end{rem}

\section{Dynamic integral conditions\label{sec6}}

We conclude our article by discussing the heat or wave equations complemented with the condition
$$\begin{pmatrix}
\mu_0(u(t))\\ \mu_1(u(t))
\end{pmatrix}\in Y,\qquad t\ge 0,$$
along with the dynamic-type one 
\begin{equation}
\label{dynmu1}
\frac{d}{dt}\begin{pmatrix}
\mu_0(u(t))\\ \mu_1(u(t))
\end{pmatrix}=-P_Y 
\begin{pmatrix}
\mu_0( u''(t)) +u(t,1)\\ u(t,0)-u(t,1)
\end{pmatrix}-P_Y K\begin{pmatrix}
\mu_0(u(t))\\ \mu_1(u(t))
\end{pmatrix},\qquad t\ge 0,
\end{equation}
for some subspace $Y$ of $\mathbb C^2$ and some $2\times 2$-matrix $K$.

An educated guess suggests to consider the same sesquilinear form $a_K$ defined in~\eqref{definform2}, but now defined on
$${\mathcal V}_Y:=\left\{\begin{pmatrix}
u\\ \phi
\end{pmatrix}\in V_Y\times Y: \begin{pmatrix}
\mu_0(u)\\ \mu_1(u)
\end{pmatrix}=\phi \right\}\equiv \left\{\begin{pmatrix}
u\\ \Gamma_1 u
\end{pmatrix}:u\in V_Y\right\}%\simeq V_Y
,$$
for some $2\times 2$-matrix $K$. In fact, the following holds, where we are using the spaces $H_Y$ introduced in Section~\ref{sec:bouz}. We write
\[
{\mathcal H}:=H^{-1}(T)\times \mathbb C^2
\]
and denote by$(\cdot|\cdot)_{\mathcal H}$ the canonical inner product of the Hilbert product space $\mathcal H$, i.e.,
\[
({\bf u}|{\bf v})_{\mathcal H}:=(u|v)_{H^{-1}(T)}+(\phi |\psi)_{\mathbb C^2},\qquad \forall {\bf u}=\begin{pmatrix}
u\\ \phi
\end{pmatrix},{\bf v}=\begin{pmatrix}
v\\ \psi
\end{pmatrix}\in {\mathcal H}.
\]

\begin{lemma}
Let $Y$ be a subspace of $\mathbb C^2$. Then the closure of ${\mathcal V}_Y$ in $\mathcal H$ is 
\[
{\mathcal H}_Y:=
\left\{\begin{pmatrix}
g\\ (y_0,y_1)^\top
\end{pmatrix}\in H_Y\times Y: y_0=\mu_0(g)\right\}.
\]
\end{lemma}
\begin{proof}
The proof is performed by first considering separately the cases where $Y$ is a Cartesian product.

1) Let us first consider the case of $Y=\{0\}^2$ or $Y=\{0\}\times \mathbb C$. Then the assertion can be proved letting
\[
L:=\Gamma_1,\qquad X_1:=V_Y,\qquad X_2:=H,\qquad Y_1=Y_2=Y.
\]
Observe that $L:V_Y\to Y$ is a bounded and (by Remark~\ref{surjmu}) surjective operator. By Corollary~\ref{cor2.5}
\[
\ker L=\left\{g\in V_Y: \mu_0(g)=\mu_1(g)=0\right\}\equiv V
\]
is dense in $X_2=H$, thus the claim follows directly from~\cite[Lemma~5.6]{MugRom06}.

2) Let now $Y={\mathbb C}\times \{0\}$ or $Y=\mathbb C^2$.
This time, the assertion can be proved applying~\cite[Lemma~5.6]{MugRom06} with
\[
L:{\mathbb V}_Y\ni \begin{pmatrix}
g\\ y_0
\end{pmatrix}\mapsto \mu_1(g)\in \mathbb C.
\]
and
\[
X_1:={\mathbb V}_Y,\quad X_2:={\mathbb H}_1,\quad Y_1=Y_2=\{0\} \hbox{ or }=\mathbb C,
\]
respectively, where
\[
{\mathbb V}_Y:=\left\{\begin{pmatrix}
g\\ y_0
\end{pmatrix}\in V_Y\times {\mathbb C}: y_0=\mu_0(g)\right\}\simeq V_Y,\qquad 
{\mathbb H}_1:=\left\{\begin{pmatrix}
g\\ y_0
\end{pmatrix}\in H^{-1}(T)\times {\mathbb C}: y_0=\mu_0(g)\right\}\simeq H^{-1}(T).
\]
The identifications are performed with respect to the isomorphism
\[
\begin{pmatrix}
g\\ \mu_0(g)
\end{pmatrix}\mapsto g.
\]
In fact, ${\mathcal V}_Y$ can be written as
\[
{\mathcal V}_Y=\left\{\begin{pmatrix}
x\\ y
\end{pmatrix}\in {\mathbb V}_Y\times \mathbb C:Lx=y
\right\}
\]
and we also have that
\[
\left\{\begin{pmatrix}
g\\ (y_0,y_1)^\top
\end{pmatrix}\in H^{-1}(T)\times Y: y_0=\mu_0(g)\right\}={\mathbb H}_1\times \left\{
\begin{array}{l}
\{0\},\quad \hbox{or}\\
\mathbb C,
\end{array}
\right.
\]
respectively. Also in this setting $L$ is a bounded and surjective operator, and moreover
\[
\ker L=\left\{\begin{pmatrix}
g\\ y_0
\end{pmatrix}\in L^2(0,1)\times {\mathbb C}: y_0=\mu_0(g),\; \mu_1(g)=0\right\}\simeq V_{\mathbb C\times \{0\}}
\]
is dense in $X_2\simeq H^{-1}(T)$ by Corollary~\ref{cor2.5}.

3) Finally, let $Y$ be the subspace spanned by a vector
$\begin{pmatrix}
1\\ \alpha
\end{pmatrix}$
for some $\alpha \not=0$. Then
\begin{eqnarray*}
\left\{\begin{pmatrix}
g\\ (y_0,y_1)^\top
\end{pmatrix}\in H^{-1}(T)\times Y: y_0=\mu_0(g)\right\}&=&
\left\{\begin{pmatrix}
g\\ (y_0,y_1)^\top
\end{pmatrix}\in H^{-1}(T)\times {\mathbb C}^2: y_0=\mu_0(g), y_1=\alpha y_0\right\}\\
&\simeq&\left\{\begin{pmatrix}
g\\ y_0
\end{pmatrix}\in H^{-1}(T)\times {\mathbb C}: y_0=\mu_0(g)\right\}\\
&\simeq& H^{-1}(T)
\end{eqnarray*}
whereas with the notation introduced in 2)
\begin{eqnarray*}
{\mathcal V}_Y 
&=& 
\left\{ 
\begin{pmatrix}
u\\ \phi 
\end{pmatrix}\in V_Y\times \mathbb C^2:\Gamma_1(u)=\phi \in Y
\right\}\\
&=& 
\left\{ 
\begin{pmatrix}
x\\ y
\end{pmatrix}\in {\mathbb V}_Y\times \mathbb C:Lx=y
\right\}\\
&\simeq & {\mathbb V}_Y\simeq V_Y.
\end{eqnarray*}
Since $V_Y$ is dense in $H^{-1}(T)$, the claim follows.
\end{proof}

As in the previous sections the form $a_K$ with domain ${\mathcal V}_Y$ is bounded and $\mathcal H_Y$-elliptic. It satisfies the Crouzeix condition. It is accretive (resp., coercive) if both eigenvalues of $K$ have positive (resp., strictly positive) real part. It is symmetric if and only if $K$ is hermitian. Hence we obtain the following result.
\begin{prop}
Let $Y$ be a subspace of $\mathbb C^2$ and let $K$ be a $2\times 2$-matrix. Then the operator $-{\mathcal A}_{Y,K}$ associated with the form generates an analytic semigroup $(e^{-t{\mathcal A}_{Y,K}})_{t\ge 0}$ of angle $\frac{\pi}{2}$ on ${\mathcal H}_Y$ and also a cosine operator function with phase space ${\mathcal V}_Y\times {\mathcal H}_Y$. This semigroup is contractive (resp., exponentially stable) if both eigenvalues of $K$ have positive (resp., strictly positive) real part.
\end{prop}
Again, $(e^{-t{\mathcal A}_{Y,K}})_{t\ge 0}$ is immediately of trace class for all $Y$ and $K$.
\begin{rem}
Let $Y$ be a subspace of $\mathbb C^2$ and let $K$ be a hermitian $2\times 2$-matrix. Then both operators $A_{Y,K}$ and ${\mathcal A}_{Y,K}$ are self-adjoint, and it follows by a direct application of Courant's minimax formula that for all $k\in \mathbb N$ the $k^{\rm th}$ eigenvalue of $A_{Y,K}$ is at least as large as the $k^{\rm th}$ eigenvalue of ${\mathcal A}_{Y,K}$.
\end{rem}

Hence, it only remains to identify the operator ${\mathcal A}_{Y,K}$. If $Y=\{0\}^2$, then it is apparent that the above conditions reduce to those in~\eqref{nlbc}, which have already been fully discussed in Section~\ref{sec:homog}. Otherwise, the following result holds.
We omit the technically involved proof, which can be performed along the lines of Theorem~\ref{identK3}.

\begin{theo}\label{identK4}
Let $Y$ be a subspace of $\mathbb C^2$, $Y\not=\{0\}^2$, and let $K$ be a $2\times 2$-matrix. Then the operator matrix ${\mathcal A}_{Y,K}$ associated with the form $a_K$ with domain ${\mathcal V}_Y$ is given by
\begin{eqnarray*}
D({\mathcal A}_{Y,K})&=&\left\{\begin{pmatrix}
u\\ \Gamma_1 u
\end{pmatrix}\in {\mathcal V}_Y: u\in H^1(0,1)	\right\},\\
{\mathcal A}_{Y,K}&=&
\begin{pmatrix}
-\Id _m^{-1} \frac{d^2}{dx^2} -\delta_1 c(\cdot) & 0\\
P_Y\begin{pmatrix}
\delta_1(\cdot) + c(\cdot)\\ 
\delta_0(\cdot) - \delta_1(\cdot)
\end{pmatrix} & P_Y K
\end{pmatrix}
\end{eqnarray*}
where $c:H^1(0,1)\to \mathbb C$ is a bounded linear functional defined by
\begin{equation}\label{cyuniq}
c(u):=\left\{
\begin{array}{ll}
0, & \hbox{if }Y=\{0\}\times {\mathbb C}^2,\\
-\left(1+\left(P_Y\begin{pmatrix}1\\0\end{pmatrix}\Big|\begin{pmatrix}1\\0\end{pmatrix}\right)_{\mathbb C^2}\right)^{-1}
\left(\left(K\Gamma_1 u+
\begin{pmatrix}
u(1) \\
u(0)-u(1)
\end{pmatrix}\Big| \begin{pmatrix}
1 \\
0
\end{pmatrix}\right)_{\mathbb C^2}\right), \qquad &\hbox{otherwise}.
\end{array}
\right.
\end{equation}
\end{theo}

Reasoning as in Theorem~\ref{identVAK} we see that whenever 
\[
\begin{pmatrix}
u\\ \Gamma_1 u
\end{pmatrix}\in D({\mathcal A}_{Y,K})
\]
with $u\in H^2(0,1)$, then $\mathcal A_{Y,K}$ acts on its first component as the second derivative. Thus, the following holds.
\begin{cor}\label{againclassic}
Let $K$ be hermitian and $Y$ be a subspace of $\mathbb C^2$. Then the semigroup $(e^{-t{\mathcal A}_{Y,K}})_{t\ge 0}$ on ${\mathcal H}_Y$ leaves ${\mathcal V}_Y$ invariant and its restriction is a semigroup on ${\mathcal V}_Y$ that is analytic of angle $\frac{\pi}{2}$ and immediately of trace class. Its generator is the part ${\mathcal A}_{Y,K}^{{\mathcal V}_Y}$ of ${\mathcal A}_{Y,K}$ in ${\mathcal V}_Y$, which is explicitly given by
\begin{eqnarray*}
D({\mathcal A}_{Y,K}^{{\mathcal V}_Y})&=&\left\{\begin{pmatrix}u\\ \phi\end{pmatrix}\in D({\mathcal A}_{Y,K}):u\in H^2(0,1)\right\},\\
A_{Y,K}^{V_Y}&=&
\begin{pmatrix}
-\frac{d^2}{dx^2} & 0\\
P_Y(K \Gamma_1 + \Gamma_2) & 0
\end{pmatrix}.
\end{eqnarray*}
\end{cor}

Now, by standard perturbation results one may deduce that in fact our heat equation is governed by an analytic semigroup on ${\mathcal H}_Y$ even if we replace~\eqref{dynmu1} by the more general condition 
$$\frac{d}{dt}\begin{pmatrix}
\mu_0(u(t))\\ \mu_1(u(t))
\end{pmatrix}=-Qu(t),\qquad t\ge 0,$$
for any bounded linear operator $Q:H^1(0,1)\to Y$. 

\begin{cor}\label{inhomogsemig}
Let $Y$ be a subspace of $\mathbb C^2$, $Y\not=\{0\}^2$. 
Let $R$ be a bounded linear operator on $H_Y$ 
and $Q$ be
\begin{itemize}
\item an arbitrary bounded linear operator from $H^1(0,1)$ to $Y$, if $Y=\{0\}\times \mathbb C$; or else
\item an arbitrary bounded linear operator from $H^1(0,1)$ to $Y$ such that 
\[
\left(Qu\Big|\begin{pmatrix}1 \\ 0\end{pmatrix} \right)= \mu_0( Ru),\quad \forall u\in H^1(0,1),
\]
if $Y\not=\{0\}\times \mathbb C$.
\end{itemize}
Then the operator matrix
\begin{eqnarray*}
D({\mathcal B})&:=&\left\{\begin{pmatrix}
u\\ \phi
\end{pmatrix}\in {\mathcal V}_Y: u\in H^1(0,1)\right\},\\
{\mathcal B}&:=&
\begin{pmatrix}
-\Id _m^{-1}\frac{d^2}{dx^2}+R & 0\\
Q & 0
\end{pmatrix},
\end{eqnarray*}
generates an analytic semigroup on ${\mathcal H}_Y$.
\end{cor}

\begin{proof}
It suffices to write $\mathcal B$ as
$${\mathcal B}={\mathcal A}_{Y,K}+{\mathcal B}_0+{\mathcal B}_1,$$
where ${\mathcal B}_0:D({\mathcal A}_{Y,K})\to {\mathcal H}_Y$ is the relatively compact operator defined by
\[
{\mathcal B}_0
\begin{pmatrix}u\\ \Gamma_1 u
\end{pmatrix}:=\begin{pmatrix}
\delta_1 c(u)\\
-P_Y\left(K\Gamma_1 u+
\begin{pmatrix}
u(1)+c(u) \\
u(0)-u(1)
\end{pmatrix}\right)
\end{pmatrix},
\]
and $\mathcal B_1$ is a bounded linear operator on $\mathcal H_Y$ 
 defined by
\[
\mathcal B_1:=\begin{pmatrix}R & 0\\ Q & 0\end{pmatrix}.
\]
Now, $Q$ has finite dimensional range and hence $\mathcal B_1$ is relatively compact. 
Accordingly, it suffices to apply a well-known perturbation result~\cite{DesSch88} to conclude.
\end{proof}

\begin{rem}\label{rem:inhomog}
In the literature, inhomogeneous integral conditions of the form
$$\int_0^1 u(t,x)\; dx=E_1(t),\qquad \int_0^1 (1-x)u(t,x)\; dx=E_2(t),\qquad t> 0,$$
are often considered. 
Reasoning as in the proof of Lemma~\ref{DA2} we see that the first coordinate of any element of $D({\mathcal A}_Y^2)$ lies in particular in $H^2(0,1)$, hence the semigroup generated by ${\mathcal A}_Y^2$ maps immediately into functions whose first coordinate does not only satisfy the generalized heat equation~\eqref{heatidm}, but also the classical one.
In view of the known semigroup approach to evolution equations with inhomogeneous conditions (see e.g.~\cite[\S6.2]{AreBatHie01}), Corollary~\ref{inhomogsemig} allows us to consider general inhomogeneous conditions of the form
$$P_{Y^\perp}\Gamma_1(u(t))=E(t),\qquad t> 0,$$
along with~\eqref{heat2} for the inhomogeneous heat equation
$$\frac{\partial u}{\partial t}(t,x)=\frac{\partial^2 u}{\partial x^2}(t,x)+\psi(t),\qquad t> 0,\; x\in (0,1),$$
and to prove well-posedness of the associated evolution equation, provided some smoothness of the time dependence of $\psi,E$ is assumed ($\psi \in L^1({\mathbb R}_+,H_Y),E\in W^{1,1}({\mathbb R}_+,Y^\perp)$ will do for a mild solution, and $\psi \in W^{1,1}({\mathbb R}_+,H_Y),E\in W^{2,1}({\mathbb R}_+,Y^\perp)$ will even yield a classical solution, cf.~\cite[\S3]{KraMugNag03}). In fact, in view of Corollary~\ref{cor:specgener} the operator $A_{Y,K}$ is invertible as long as $K$ is positive semidefinite and therefore we can even write down an explicit formula for the solution (in dependence of the semigroup generated by $A_{Y,K}$), cf.~\cite[Prop.~3.9]{KraMugNag03}.
\end{rem}

\noindent{\bf Acknowledgement.}
The first author is partially supported by the Land Baden--W\"urttemberg in the framework of the \emph{Juniorprofessorenprogramm} -- research project on ``Symmetry methods in quantum graphs''. 

This article has been partly written while the first author was visiting Valenciennes. He would like to thank the University of Valenciennes for the hospitality and for partial financial support.
Both authors thank D. Mercier (Lamav, University of Valenciennes) for his help in the proof of Corollary~\ref{cWeylgenY}. They are grateful to the anonymous referee for careful reading and for many useful suggestions, both on form and content.

\bibliographystyle{abbrv}

\end{document}